\documentclass{article} 

\usepackage{graphicx} 
\usepackage[a4paper, total={6.5in, 9.5in}]{geometry} 
\usepackage{amsmath,amsthm,amssymb}
\usepackage{mathtools}
\usepackage[dvipsnames,table]{xcolor} 
\usepackage{setspace}
\usepackage[para,online,flushleft]{threeparttable}
 \usepackage{booktabs}
 \usepackage{natbib}
 \usepackage{subcaption}  
\usepackage{url}
\usepackage{pifont}
\usepackage{amsfonts}
\usepackage{enumitem}
\usepackage{siunitx}
\usepackage{tabularx}

\newtheorem{theorem}{Theorem}[section]

\theoremstyle{definition}
\newtheorem{example}{Example}

\usepackage{hyperref}
\hypersetup{
    colorlinks=true,
    citecolor=blue,
    linkcolor=blue
    }
\usepackage{cleveref} 

\usepackage{tikz}
\usetikzlibrary{backgrounds}
\usetikzlibrary{calc}
\usetikzlibrary {arrows.meta}
\usetikzlibrary{decorations.pathmorphing}
\usetikzlibrary{fit,tikzmark}
\usetikzlibrary{positioning}

\tikzset{
  default/.style={circle, draw, inner sep=1pt, minimum size=5mm, font=\small},
}

\definecolor{color1}{RGB}{228,26,28}    
\definecolor{color2}{RGB}{55,126,184}   
\definecolor{color3}{RGB}{77,175,74}    
\definecolor{color4}{RGB}{152,78,163}   
\definecolor{color5}{RGB}{255,127,0}    
\definecolor{color6}{RGB}{255,105,180} 
\definecolor{color7}{RGB}{166,86,40}    

\tikzset{
  edge1/.style={color1,thick,solid},
  edge2/.style={color2,thick,solid},
  edge3/.style={color3,thick,solid},
  edge4/.style={color4,thick,solid},
  edge5/.style={color5,thick,solid},
  edge6/.style={color6,thick,solid},
  edge7/.style={color7,thick,solid},
  fictive_edge/.style={thick, dotted},
  edge1blurred/.style={color1,thick,solid,opacity=0.3},
  edge2blurred/.style={color2,thick,solid,opacity=0.3},
  edge3blurred/.style={color3,thick,solid,opacity=0.3},
  edge4blurred/.style={color4,thick,solid,opacity=0.3},
  edge5blurred/.style={color5,thick,solid,opacity=0.3},
  edge6blurred/.style={color6,thick,solid,opacity=0.3},
  edge7blurred/.style={color7,thick,solid,opacity=0.3},
}

\definecolor{acadpurple}{RGB}{128,0,128}
\definecolor{acadcyan}{RGB}{0,150,200}
\definecolor{acadgreen}{RGB}{0,120,0}

\newcommand{\typeA}[1]{\fcolorbox{black}{acadpurple!20}{#1}}
\newcommand{\typeB}[1]{\fcolorbox{black}{acadcyan!60}{#1}}
\newcommand{\typeC}[1]{\fcolorbox{black}{acadgreen!40}{#1}}
\newcommand{\typeD}[1]{\fcolorbox{white}{white}{#1}}

\newenvironment{keywords}{
    \vspace{1ex} 
    \noindent\textbf{Keywords:} 
}{\par\vspace{1ex}} 

\usepackage[textwidth=2cm]{todonotes}

\begin{document}
    
\title{The incomplete Traveling Tournament Problem} 
\author{Karel Devriesere, David Van Bulck, Dries Goossens} 
\date{} 

\maketitle

\begin{abstract}
    We present a new problem called the incomplete Traveling Tournament problem, which introduces the well known Traveling Tournament Problem into the realm of incomplete round-robin tournaments. We focus on the case where teams can face each opponent at most once. We give a formal description of this problem and show that it is $\mathcal{NP}$-hard. We first discuss how we can obtain lower bounds and how to strengthen them. Then, we propose two integer programming formulations and compare their LP-relaxations. We also propose a third formulation that assumes that home-away patterns of teams are fixed. 
    We discuss how a recently proposed metaheuristic for incomplete round-robin scheduling can be tailored to our problem.
    In doing so, we present a novel neighborhood structure and show it fully connects the home-away pattern solution space.
    Finally, problem instances are proposed, for which we derive lower and upper bounds. We show that these instances are challenging, making the development of efficient algorithms for the incomplete Traveling Tournament problem an interesting direction for future research. 
\end{abstract}


\section{Introduction}

In an incomplete round-robin tournament (iRR), the number of available rounds is too small for a team to play all its possible opponents. As a result, while all teams play the same number of games, only a subset of their potential games is scheduled, with the selection determined by the competition organizer (possibly by a draw).
Recently, the iRR tournament format has gained popularity, especially since the Union of European Football Associations (UEFA) decided to reform the group stages of their most prestigious competitions (Champions league, Europa League, and the Conference League) to an iRR (see e.g.\ \citet{Guyon2025}, \citet{CsatoDevriesere2025}). In this case, each of the 36 teams faces just eight (or six in case of the Conference League) of their 35 opponents in the league. 
Also in the United States, competitions organized into conferences and/or divisions often give rise to incomplete round-robin (iRR) tournaments. For instance, in Major League Soccer, teams face only six of the fifteen possible opponents from the opposite conference, with the selected games announced prior to the season as part of the official timetable. 
The iRR format is nowadays also used in youth sports competitions where there are tens if not hundreds of teams within the same age or strength category that form potential opponents. While a classic round-robin (RR) tournament would simply take too long to play all games, the incomplete round-robin (iRR) format is an appealing alternative (see also \citet{li2025beyond,devriesere2025redesigning}).

One of the most prominent objectives studied in sports scheduling is the minimization of the total distance traveled by all teams \citep{ribeiro2012sports}.
We distinguish two cases.
In the first case, particularly popular in European competitions, teams return home after each away game.
Consequently, in a classic RR, the travel distance is fully determined once the participating teams in the tournament have been decided (see also \citet{Toffolo2017}) and the schedule has no impact on the total distance traveled (unless the scope is limited to a subset of the rounds like holiday periods; see e.g.\ \citet{Kendall2008}).
More interesting, though, is the iRR case, where the schedule can significantly impact the distance traveled trough careful determination of match pairings. 
This setting has been intensively studied by \citet{li2025beyond} and \citet{devriesere2025redesigning} for the case of Belgian youth football and field hockey, respectively.
As shown by \citet{li2025beyond}, the iRR format can substantially reduce total travel distance compared to a format in which teams are partitioned into leagues, each of which is organized as a RR.
In the second case, teams travel directly between their opponents’ venues when consecutively playing away, forming a so-called \emph{road trip}.
In the RR format, the problem of minimizing travel distance is known as the Traveling Tournament Problem (TTP) and has been studied extensively in the literature (see, e.g.\ \citet{easton2001traveling,zhao2025matching,uthus2012solving,miyashiro2012approximation,goerigk2014solving}).

This paper starts by generalizing the TTP such that it accommodates a broad range of real-life tournament formats, including the iRR. In the \emph{Generalized Traveling Tournament Problem} (GTTP), matches are selected from a predefined set of possible games and assigned to a set of rounds such that each team plays the same number of games (each team plays once per round). For every team, the number of home and away games must be as balanced as possible. In addition, constraints limit how frequently any pair of teams may face each other, as well as the maximum number of consecutive home or away games allowed for each team. The objective is to minimize the total travel distance, assuming road trips. 
A formal description of the Generalized TTP (GTTP) is given in \Cref{sec:GTTP}.

Then, we focus on a special case of GTTP, assuming there are not enough rounds to play all games.
Hence, in this case, travel distances are influenced not only by effectively grouping away games into road trips, but also by the choice of matches. 
We coin this special case the \emph{incomplete Traveling Tournament Problem} (iTTP). 
In particular, we assume that the number of rounds is even and that every team can face every other team, but at most once over the tournament. 
Both assumptions are common in sports scheduling and result from fairness considerations. Indeed, an odd number of rounds implies that some teams have more home games than others, offering them an unfair advantage (this has been shown for iRR chess tournaments by \citet{csato2024most}). Moreover, allowing a team to face some opponents multiple times while not all opponents have been played would exacerbate imbalances in the strength of opponents across teams.
Like the TTP, the proposed iTTP is a judicious simplification of real-world scheduling problems, which typically involve dozens, if not hundreds, of additional constraints (see \citet{VanBulck2019}). It focuses on the most fundamental elements of the problem, while still capturing the sources of its complexity. The iTTP is formally introduced in \Cref{sec:iTTP}, together with computational complexity results. 

In \Cref{sec:LowerBounds}, we develop lower bounds for iTTP, discuss ways to strengthen them, and contrast them with TTP lower bounds from the literature. In \Cref{sec:UpperBounds}, we develop solution strategies that build directly on the literature, including integer programming formulations, together with a heuristic. We first show that recently proposed neighborhood structures for modifying home-away patterns are not connected. Then, we propose a new neighborhood structure, which involves swapping the home-away status of two teams in two rounds. We show that this neighborhood is connected. In \Cref{sec:Results}, we test the performance of the bounds and solution strategies on a problem instance set derived from the TTP literature. It turns out that iTTP remains challenging to solve, and solution strategies for solving TTP or constructing iRRs are not necessarily applicable or suitable for iTTP. 

\section{Problem description}\label{sec:ProblemDescription}

We first describe the Generalized Traveling Tournament Problem (\Cref{sec:GTTP}), and then formally introduce the Incomplete Traveling Tournament Problem (\Cref{sec:iTTP}).

\subsection{Generalized Traveling Tournament Problem}\label{sec:GTTP}

Here, we give a formal description of GTTP, which generalizes the well known TTP.
Given is a set of teams $T =\{ 1,\dots,n \}$ and a set of rounds $R = \{ 1,\dots,r \}$. Each team has a home venue, and a game is played at the venue of one of the opposing teams. The team that plays home is called the home team while the visiting team is called the away team. Also given is a set $M$ of tuples such that $(i,j)$ represents the match between teams $i$ and $j$, $i,j \in T: i \neq j$, with $i$ being the home team. The set $M$ represents the set of matches that the competition organizer can choose from. We assume that $n$ is even such that each team plays exactly one game in each round. To fairly balance the home advantage (see e.g.\ \citet{Pollard2005,Pollard2017}), it is required that each team plays an equal number of home and away games. Thus, $r$ is also assumed to be even. For every distinct pair of teams $i$ and $j$, we are given a parameter $\alpha_{ij}$ that denotes the minimum number of times that $i$ and $j$ must face each other, and a parameter $\beta_{ij}$ that denotes the maximum number of times that $i$ and $j$ can face each other (with $\alpha_{ij}=\alpha_{ji}$ and $\beta_{ij} = \beta_{ji}$). Moreover, we are given a parameter $\gamma$ which states the number of consecutive slots in which teams are allowed to face each other at most once. 

At the beginning of the tournament, all teams are at their home venue, to which they must return after their last away game. Given is a matrix $(d_{i,j})_{i,j \in T}$, such that $d_{i,j}$ is the distance between teams $i$ and $j$, with $d_{i,i} = 0$. 
A team undertakes a road trip whenever it plays one or more consecutive away games: the team departs from its home venue before the first away game in the sequence, travels directly between the venues of successive opponents, and returns to its home venue only after the final away game. To formalize travel, we define a leg of a road trip as a single direct journey between two venues. For example, if a team $l$ has three consecutive away games against $i,j$ and $k$, the total travel distance of this road trip is $d_{l,i}+d_{i,j}+d_{j,k}+d_{k,l}$, and the trip consists of 4 legs. Also given is a parameter $\lambda$ that specifies the maximum permitted length of a road trip or home stand, measured as the number of consecutive away (respectively, home) games. Consequently, a road trip will consist of at most $\lambda + 1$ legs. The objective is to minimize the overall travel distance of all road trips. Hence, this problem can formally be defined as follows:

\begin{description}[align=left]
    \item[\textbf{Generalized Traveling Tournament Problem (GTTP)}]
    \item[\textbf{Input:}] A set of teams $T=\{1,\dots,n\}$ with n even, a set of rounds $R=\{1,\dots,r\}$, with $r$ even, a set of matches $M$, a distance matrix $(d_{ij})_{i,j \in T}$, non-negative integers $\alpha_{ij}, \beta_{ij}, i,j \in T$, and integers $\lambda, \gamma \geq 1$.

    \item[\textbf{Output:}] A timetable minimizing the sum of distances traveled by all teams such that:
    \begin{description}
        \item[\textbf{C1}] A team has at most $\lambda$ consecutive home and $\lambda$ consecutive away games,
        \item[\textbf{C2}] A team has exactly one game per round,
        \item[\textbf{C3}] A team has exactly $\frac{r}{2}$ home and $\frac{r}{2}$ away games,
        \item[\textbf{C4}] Two teams $i$ and $j$ face each other at least $\alpha_{ij}$ and at most $\beta_{ij}$ times,
        \item[\textbf{C5}] Only matches from the set $M$ are included in the timetable,
        \item[\textbf{C6}] Two teams face each other at most once in any $\gamma$ consecutive rounds.
    \end{description}
\end{description} 

This problem unifies several proposed TTP variants. In the classic TTP, a double round-robin tournament should be constructed, so $r=2(n-1)$, $M = \{(i,j), i,j \in T: i \neq j\}$, and $\alpha_{ij}=\beta_{ij}=2$ for every pair $i,j \in T$. Typically, $\gamma=2$ and $\lambda=3$.
\citet{ribeiro2007heuristics} consider the mirrored TTP (mTTP), in which a team faces every other team once at home or away in the first $n-1$ rounds. The remaining rounds are an exact copy of the first half except that the home-away status of games is reversed. Therefore, $r=2(n-1), \alpha_{ij}=\beta_{ij}=2$ for every pair $i,j \in T$, and $\gamma=n-1$. In the TTP with predefined venues, proposed by \citet{melo2009traveling}, the location of each game is specified and a single round-robin tournament should be constructed (so $\alpha_{ij}=\beta_{ij}=1$). Hence, $M$ consists only of tuples $(i,j)$ if the match between $i$ and $j$ is played at the venue of $i$; for every two teams $i$ and $j$, either $(i,j)$ or $(j,i)$ is in $M$. The bipartite TTP consists of two leagues where only inter-league matches take place \citep{hoshino2011inter}, thus $M$ consists only of inter-league matches. 

Note that since the classic TTP was shown to be $\mathcal{NP}$-hard \citep{bhattacharyya2009note}, being a generalization, GTTP is clearly also $\mathcal{NP}$-hard.

\subsection{Incomplete Traveling Tournament Problem}\label{sec:iTTP}

In the incomplete Traveling Tournament Problem (iTTP), each team may face any other team at most once but there are insufficient rounds for every team to play against all others.
In other words, $M = \{(i,j): i,j \in T: i \neq j\}$, $r \leq n-2$ ($r$ is even), $\alpha_{ij} = 0$ and $\beta_{ij} = 1$, for every pair $i,j \in T$. Note that since $M$ contains all matches and teams meet each other at most once, Constraints \textbf{C5} and \textbf{C6} are always satisfied (and parameter $\gamma$ is redundant). The problem can formally be defined as follows:

\begin{description}[align=left]
    \item[\textbf{Incomplete Traveling Tournament Problem (iTTP)}]
    \item[\textbf{Input:}] A set of teams $T=\{1,\dots,n\}$ with n even, a set of rounds $R=\{1,\dots,r\}$ with $r$ even and $r \leq n-2$, a distance matrix $(d_{ij})_{i,j \in T}$, $\alpha_{ij} = 0$ and $\beta_{ij} = 1$ for all $i,j \in T: i \neq j$, and an integer $\lambda \geq 1$.

    \item[\textbf{Output:}] A timetable minimizing the sum of distances traveled by all teams, such that Constraints \textbf{C1}-\textbf{C4} are satisfied.
\end{description} 

A timetable where Constraints \textbf{C1}-\textbf{C4} are satisfied is called feasible. We now give the complexity status of iTTP. Since $r$ is even, the smallest possible tournament contains only two rounds. We show that even for this case, iTTP is $\mathcal{NP}$-hard.

\begin{theorem}
\label{thm:complexity}
The iTTP is $\mathcal{NP}$-hard, even if $r=2$.
\end{theorem}
\begin{proof}
We prove the theorem by a reduction from iRR-Eligible, which was proven to be $\mathcal{NP}$-complete by \cite{li2025beyond}. This problem is defined as follows:\\

\textbf{iRR-Eligible \citep{li2025beyond}}.\\
\textit{Input.} A set of teams $T=\{1,\dots,n\}$, where $n$ is even, a set of rounds $R=\{1,2\}$ (i.e., $r=2$), a set $M'$ such that $(i,j) \in M' \iff (j,i) \in M'$ and $\alpha_{ij} = 0, \beta_{ij} = 1$, for all $i,j \in T$. \\
\textit{Question.} Does there exist an iRR tournament satisfying Constraints \textbf{C2}, \textbf{C4} and \textbf{C5}?\\

We build an instance of iTTP by copying all teams and rounds from iRR-Eligible and setting $\lambda=1$ and the travel distance $d_{i,j}$ between two teams $i$ and $j$ to 0 in case $(i,j)\in M'$ and to 1 otherwise.

Given a solution to iRR-Eligible, we now construct a zero travel solution to iTTP.
To this end, we schedule game $(i,j)$ with $i<j$ in round $s\in R$ whenever game $(i,j)$ or $(j,i)$ is scheduled in round $s$ of the given solution.
Since iRR-Eligible takes into account \textbf{C2} and \textbf{C4}, these Constraints are also respected in the iTTP solution.
We now show how to adapt the home-away assignment of the scheduled games to ensure that Constraints \textbf{C1} and \textbf{C3} are satisfied.
For this, construct a graph $G(T,E)$ so that teams correspond to vertices and add an edge in $E$ for every game that we scheduled between any of the teams.
Since each game appears at most once in $M'$ and since each team plays precisely one game per round, observe that the games of round 1 and 2 form a collection of two edge disjoint perfect matchings (see e.g., \citet{ribeiro2023tutorial}) and hence form a collection of cycles covering all vertices (see e.g. \citep{kaski2014switching}).
For any cycle we now arbitrarily orient the cycles so that an incoming arc for any team corresponds to a home game and an outgoing arc corresponds to an away game. This results in each team playing precisely one home and one away game, thereby satisfying \textbf{C3} (and thus also \textbf{C1}). Moreover, since teams return home after every away game, the total travel time induced by selecting the matches in the solution for iRR-Eligible is equal to 0. Hence, the existence of a solution for iRR-Eligible implies the existence of a solution to iTTP with a total travel distance of 0.

For the other direction, if iTTP has a solution with a total travel time of 0, this means that there exist 2 edge disjoint perfect matchings in the iRR-Eligible instance where only eligible teams play each other. Therefore, \textbf{C5} is satisfied for iRR-Eligible. Since any solution for iTTP satisfies \textbf{C2} and \textbf{C4}, the existence of a solution for iTTP with a total travel distance of 0 implies the existence of a solution for iRR-Eligible.
\end{proof} 

Although optimizing iTTP is hard, we now show that finding a feasible solution is easy and can be done in polynomial time (we will exploit this result later in the paper). For this purpose, we recall the concept of home-away patterns (HAPs). If we state for each round whether a team plays home or away, we get a sequence of H's (Home) and A's (Away); such a sequence is known as a HAP. Two HAPs $h_1$ and $h_2$ are called complementary if $h_2$ has a home game on each round where there is an away in $h_1$, and visa versa. A HAP is called \emph{balanced} if it contains an equal number of home and away games. HAPs are induced by the timetable, but a popular decomposition strategy is to first assign HAPs to teams from a so called HAP-set $\mathcal{H}$, and only then to decide on the matches played in each round. This decomposition strategy is known as first-break then schedule (see e.g. \citet{nemhauser1998scheduling}). A pair of consecutive home or away games is called a break. 

In case that $r \leq \frac{n}{2}$, \citet{de2026break} show that it is possible to construct a timetable satisfying Constraints~\textbf{C1}-\textbf{C4} by taking the two (complementary) HAPs with zero breaks, such that half of the teams is assigned the HAP that starts with a home game and the other half of the teams is assigned the HAP that starts with an away game. Observe that it is impossible to construct a timetable when $\lambda=1$ and $r > \frac{n}{2}$ without violating the Constraint that two teams can face each other at most once.

We now propose an approach that always results in a feasible timetable for iTTP in case that $\lambda \geq 2$ and the number of rounds is arbitrary (but smaller than $n-1$). The first step involves constructing a single round-robin schedule, without specifying the home team of each match yet. This can easily be done with the circle method \citep{de1981scheduling} or Vizing's edge coloring algorithm \citep{Januario2016}. 
Next, we take the first $r$ rounds of this schedule and partition the $r$ rounds into consecutive pairs. Recall from the proof of \Cref{thm:complexity} that since any two rounds in a feasible schedule correspond to two edge-disjoint one-factors, and two edge-disjoint one-factors form a collection of even length cycles, we can set the orientations for each cycle for each consecutive pair of rounds such that each team has one home and one away game in this pair of rounds. Hence, any team will have at most two consecutive home games (and  at most two consecutive away games) in its HAP, and we have constructed a feasible timetable for any iTTP instance. 

\section{Lower bounds}\label{sec:LowerBounds}

In this section, we discuss how to obtain lower bounds for iTTP. In \Cref{subsec:ilb}, we first discuss why a well known lower bound from TTP could be of poor quality for iTTP.
Then, we show how to strengthen this bound in \Cref{subsec:dlb,sec:coloring_constraints,subsec:mnt}.

\subsection{Independent Lower Bound}\label{subsec:ilb}

A well known lower bound for TTP is to compute the minimum travel distance for each individual team, and then sum the resulting travel distances. This is known as the Independent Lower Bound (ILB; see \cite{easton2001traveling}). Computing the minimum travel distance of an individual team $t \in T$ can be done by solving a Capacitated Vehicle Routing Problem (CVRP), where the depot corresponds to the venue of $t$, the clients correspond to the opponents of $t$ each with a demand of 1, the capacity of the vehicle is equal to $\lambda$, and each tour corresponds to a road trip (see e.g., \citet{urrutia2007new}).

In the iTTP, the ILB can be computed by solving $n$ independent instances of the Prize-Collecting Vehicle Routing Problem (PCVRP), which generalizes the classic VRP by allowing customers to be visited at most once and by requiring the collection of a minimum total prize, rather than visiting all customers exactly once (see e.g.\ \cite{long2019hybrid}). Indeed, when solving the problem for team $i\in T$, let $i$'s venue correspond to the depot, and associate with each other team in the tournament a client with demand and prize of 1 and require that the minimal prize collected amounts to $\frac{r}{2}$.
Finally, introduce $\frac{r}{2}$ vehicles with capacity of $\lambda$ each. Then, in an optimal solution to the associated PCVRP instance, at least $\frac{r}{2}$ different teams are visited and road trips of length at most $\lambda$ are used. Hence, Constraints \textbf{C1} and \textbf{C4} are satisfied, thus providing a valid lower bound for iTTP. Note that if the triangle inequality is satisfied, exactly $\frac{r}{2}$ customers will be visited, satisfying Constraints \textbf{C3} concerning away games.

The ILB has been shown to be strong for most TTP instances (within 5$\%$ of the best known solution, see e.g.\ \citet{miyashiro2012approximation}). However, the ILB for iTTP  ignores that any two teams are allowed to play at most once, and that each team must be visited precisely $\frac{r}{2}$ times.  An example of this is shown in \Cref{fig:ilb_CIRC12_8}, where it can be seen that each team visits its two nearest neighbors. However, if two teams are each others nearest neighbor, this implies that they meet each other twice. Moreover, remote teams that are not among the nearest neighbors of any of the other teams are never visited, thereby violating Constraints \textbf{C3}. All together, these violations could have a major impact on the quality of the ILB.

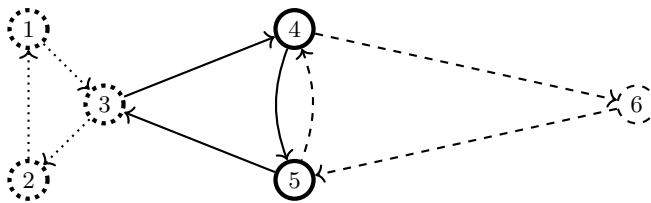
\begin{figure}[h!]
\caption{Example of the ILB for an iTTP with 6 teams and 4 rounds. The optimal travel distance of dotted teams is when they mutually travel to each other. The optimal travel distance of teams 4 and 5 is when they travel to each other and to team 3. Team 6 is not visited by any team according to the ILB.}
\label{fig:ilb_CIRC12_8}
\centering
\begin{tikzpicture}
    \node[default, ultra thick, dotted] (v1) at (-4,1) {1};
    \node[default, ultra thick, dotted] (v2) at (-4,-1) {2};
    \node[default, ultra thick, dotted] (v3) at (-3,0) {3};
    \node[default, ultra thick] (v4) at (-0.5,1) {4};
    \node[default, ultra thick] (v5) at (-0.5,-1) {5};
    \node[default, ultra thick, dashed, thick] (v6) at (4,0) {6};

    \draw[->, thick, dotted] (v1) -- (v3);
    \draw[->, thick, dotted] (v2) -- (v1);
    \draw[->, thick, dotted] (v3) -- (v2);

    \draw[->, thick] (v3) -- (v4);
    \draw[<-, thick,dashed, thick] (v4) to[bend left=20] (v5);
    \draw[->, thick] (v4) to[bend right=20] (v5);
    \draw[->, thick] (v5) -- (v3);

    \draw[->, thick, dashed] (v4) -- (v6);
    \draw[->, thick, dashed] (v6) -- (v5);
\end{tikzpicture}
\end{figure}

\subsection{Dependent Lower Bound}\label{subsec:dlb}

Here, we propose a stronger lower bound by considering all teams together. Specifically, we minimize the total travel distance of all teams, such that each team is visited $\frac{r}{2}$ times and any pair of teams meet at most once.
We thereby only consider the road trips that teams undertake but not the rounds in which games are scheduled. 
We call this bound the Dependent Lower Bound (DLB).

In order to compute this bound, we first enumerate all possible road trips $P_t$ for each team $t \in T$. 
Each road trip $p \in P_t$ visits $\delta_{tp}$ other teams and induces a travel distance of $c_{tp}$. 
Then, we introduce a binary variable $z_{tp}$ that is 1 if road trip $p \in P_t$ is assigned to team $t$. 
In spirit of \citet{Cheung2009}, we then use the following binary linear program to compute the DLB. 

\begin{align}
    \text{min} & \sum_{t \in T}\sum_{p \in P_t}c_{tp}z_{tp} \label{lb2:obj} \\
    \text{s.t.} & \sum_{p \in P_t}\delta_{tp}z_{tp} = \frac{r}{2} & \forall t \in T \label{lb:c1}\\
    & \sum_{i \in T \setminus \{t\}}\sum_{p \in P_i: t \in p}z_{ip} = \frac{r}{2} & \forall t \in T \label{lb:c2}
\end{align}
\begin{align}
    & \sum_{p \in P_t: i \in p}z_{tp} + \sum_{p \in P_i: t \in p}z_{ip} \leq 1 & \forall t,i \in T: t < i \label{lb:c3} \\
    & z_{tp} \in \{0,1\} & \forall t \in T, \ \forall p \in P_t \label{lb2:c6}
\end{align}

The objective function~\eqref{lb2:obj} minimizes the total cost of the road trips of all the teams. Constraints~\eqref{lb:c1} state that teams should play exactly $\frac{r}{2}$ away games, while Constraints~\eqref{lb:c2} state that every team should be present in exactly $\frac{r}{2}$ road trips of other teams, therefore enforcing that every team plays $\frac{r}{2}$ home games. Next, Constraints~\eqref{lb:c3} guarantee that two teams face each other at most once. Finally, Expressions~\eqref{lb2:c6} give the variable domains.


\subsection{One-factorizable Dependent Lower Bound}\label{sec:coloring_constraints}

While the DLB ensures that each team plays $r$ games, it does not ensure that these games can be scheduled within $r$ rounds. If we construct a graph where each node corresponds to a distinct team and two nodes are connected to each other by an edge if one is visited by the other, then this graph is r-regular. 
For a feasible timetable to exist, this graph must be $r$-edge-colorable, where the colors represent the rounds (see also \citet{Guyon2025}).
In case $r<\frac{n}{2}$, it is well known that not every $r$-regular graph is $r$ colorable (i.e., one-factorable), and hence we can strengthen the DLB as follows.
Let $x_{tis}$ be a binary variable that is 1 if teams $t,i \in T$ face each other in round $s \in R$ and is 0 otherwise (with $x_{its} = x_{tis}$). Then, we can add the following Constraints: 

\begin{align}
    & \sum_{p \in P_t: i \in p}z_{tp} + \sum_{p \in P_i: t \in p}z_{ip} = \sum_{s \in R}x_{tis} & \forall t,i \in T: t \neq i \label{lb2:c9} \\
    & \sum_{i \in T: i \neq t}x_{tis} = 1 & \forall t \in T, \ \forall s \in R \label{lb2:c10} \\
    & x_{its} = x_{tis} & \forall i,t \in T: i \neq t, \ \forall s \in R \label{lb2:c11} \\
    & x_{its} \in \{0,1\} & \forall i,t \in T: i \neq t, \ \forall s \in R \label{lb2:c12}
\end{align}

Constraints \eqref{lb2:c9} state that two teams face each other if one is included in the other's road trip. Constraints \eqref{lb2:c10} impose that, in every round, each team faces exactly one opponent. Together, they enforce that a feasible schedule (i.e. satisfying Constraint \textbf{C2}) can be constructed with the set of matches implied by the road trips. However, it ignores whether a feasible schedule can be constructed with the actual road trips (e.g. if $t$ first visits $i$ and then $j$ in a road trip, then $x_{tis} = x_{tjs+1} = 1$ for some $s \in R \setminus \{r\}$), and it does not take into account that teams can play at most $\lambda$ consecutive home-and away games. We refer to the DLB extended with Constraints~\eqref{lb2:c9}-\eqref{lb2:c12} as DLB-1F.

At the same time, it is known that any $r$-regular graph is one-factorable if $r\geq \frac{6n}{7}$. Moreover, the infamous one-factorization conjecture asserts that if $r \geq \frac{n}{2}$, the resulting graph is one-factorable (see also \citet{chetwynd1985regular}; this conjecture was proven to be true when $n$ is sufficiently large by \citet{csaba2016proof}). Hence, adding Constraints~\eqref{lb2:c9}-\eqref{lb2:c12} is only likely to strengthen the bound in case that $r < \frac{n}{2}$.

\subsection{Minimum Number of Legs Dependent Lower Bound}
\label{subsec:mnt}

Rather than total travel distance, \citet{urrutia2006maximizing} minimize the number of legs $D(S)$ in timetable $S$. This is particularly relevant when teams travel by air or when distances between all pairs of teams are identical.
The authors show that minimizing legs is equivalent to maximizing breaks. Similarly, we now show that we can express the overall number of legs in terms of the number of road trips in timetable $S$, which we denote by $L(S)$:

\begin{equation}\label{eq:road_trips}
    D(S) = \frac{nr}{2}+L(S).
\end{equation}

Indeed, each team plays precisely $\frac{r}{2}$ away games and each away game induces a leg. After the last away game of a road trip, there is one more leg in order for the team to return home. As a result of Equation~\eqref{eq:road_trips}, minimizing the total number of legs is equivalent to minimizing the number of times teams start a road trip. Since a road trip contains at most $\lambda$ away games, each team needs at least $\lceil\frac{r}{2\lambda}\rceil$ road trips.
If $r \leq \frac{n}{2}$ a timetable with exactly this number of trips also exist.

\begin{theorem}\label{thm:travel_distance_CON}
    The minimum total number of legs in an iTTP instance with $n$ teams, $r\leq \frac{n}{2}$ rounds, is equal to $\frac{nr}{2} + n\lceil \frac{r}{2\lambda} \rceil$.
\end{theorem}
\begin{proof}
    We first construct two complementary HAPs with $\lceil \frac{r}{2\lambda} \rceil$ legs each.
    To this purpose, let H$^\lambda$ and A$^\lambda$ denote a sequence of $\lambda$ consecutive home and away games, respectively. Define $j := \frac{r}{2}\bmod{}\lambda$. If $j = 0$, then the two patterns consecutively take either H$^\lambda$A$^\lambda$ or A$^\lambda$H$^\lambda$.
    If $j \geq 1$, the two patterns consecutively take either H$^\lambda$A$^\lambda$ or A$^\lambda$H$^\lambda$ for the first $r-2j$ rounds, followed by H$^j$A$^j$ in case of the first HAP and A$^j$H$^j$ in case of the complementary HAP.

    It follows from \citet[Section 2.]{de2026break} that in case that $r \leq \frac{n}{2}$, we can always construct a feasible timetable by arbitrarily assigning half of the teams to one HAP and all other teams to the complement of this HAP. 
    The bound in this theorem thus follows immediately from Equation~\eqref{eq:road_trips}.
\end{proof}

Let $v_{nr}$ be a lower bound on the minimal number of legs for a tournament with $n$ teams and an arbitrary number of rounds $r$. 
In spirit of \citet{urrutia2007new} for the classic TTP, we now define the minimum number of legs dependent lower (DLB-MinLeg) bound by adding the following Constraint to the formulation of DLB.
\begin{equation}\label{lb2:c8}
    \sum_{t \in T}\sum_{p \in P_t}z_{tp} \geq v_{nr} 
\end{equation}

In case $r\leq \frac{n}{2}$, \Cref{thm:travel_distance_CON} shows that enforcing that no team plays more than $\lambda$ consecutive home or away games is sufficient to enforce that Constraint \eqref{lb2:c8} is met and hence that DLB=DLB-MinLeg.
In the other case that $r> \frac{n}{2}$, we compute a bound on $v_{nr}$ by constructing an incomplete round robin tournament that maximizes the overall number of breaks in the timetable.
For this, we apply a Benders' decomposition approach where the master problem constructs a break maximal HAP set while the sub problem checks feasibility of this HAP set.
For detailed documentation of this approach, we refer to \citet{li2025beyond} and \citet{VanBulck2023}.

\section{Solution strategies}\label{sec:UpperBounds}

In this section, we discuss strategies to obtain solutions for iTTP instances. In \Cref{sec:UpperBounds_IP}, we present a straightforward integer programming formulation  while in \Cref{sec:HAPdriven}, we describe a formulation that uses road trip variables. The linear programming (LP) relaxations of these models also yield lower bounds, and will later be compared to the lower bounds presented in \Cref{sec:LowerBounds}. In \Cref{sec:HAPconstrained}, we construct a formulation that assumes the HAP of teams is fixed. Finally, in \Cref{sec:UpperBounds_greedy_matching}, we discuss a greedy matching algorithm inspired by earlier work on iRR timetabling. 

\subsection{A straightforward  formulation}\label{sec:UpperBounds_IP}

Let $x_{ijs}$ be a binary variable that is 1 if team $i$ plays home against team $j$ in round $s$ and is 0 otherwise, together with the variable $y_{tij}$ that is 1 if team $t$ travels from $i$ to $j$ at some point in the tournament, and is 0 otherwise. We also define the set $R_{s} = \{1,\dots,s\}$ for any $s \in R$. The following formulation, which we label $\mathcal{F}_1$, is similar to the one given in \citet[Section 2.1.]{melo2009traveling}, except that we now enforce an iRR instead of a double round format, and that \citet{melo2009traveling} assume that the venue of each game is predetermined. Similar formulations using the $x$-variables defined here can be found in, for example, \citet{goerigk2016combined} and \citet{duran2019scheduling}.

\begin{align}
    \text{min} & \sum_{t \in T}\sum_{i \in T}\sum_{j \in T} d_{ij}y_{tij} \label{IP:obj} \\
    \text{s.t.} & \sum_{s \in R}(x_{ijs}+x_{jis}) \leq 1 & \forall i,j \in T: i < j \label{IP:c1}\\
    & \sum_{j \in T \setminus \{i\}}(x_{ijs}+x_{jis}) = 1 & \forall i \in T, \ s \in R \label{IP:c2}\\
    & \sum_{j \in T \setminus \{i\}}\sum_{s \in R}x_{ijs} = \frac{r}{2} & \forall i \in T \label{IP:c3} \\
    & x_{its-1} + x_{jts} - 1 \leq y_{tij} & \forall i,j,t \in T, \ s \in R \setminus \{1\} \label{IP:c4} \\
    & x_{its-1} + \sum_{j \in T \setminus \{i\}}x_{tjs} - 1 \leq y_{tit} & \forall i,t \in T, \ s \in R \setminus \{1\} \label{IP:c5} \\
    & \sum_{j \in T \setminus \{i\}}x_{tjs-1} + x_{its} - 1 \leq y_{tti} & \forall i,t \in T, \ s \in R \setminus \{1\} \label{IP:c6} \\
    & x_{it1} \leq y_{tti} & \forall i,t \in T \label{IP:c7}\\
    & x_{itr} \leq y_{tit} & \forall i,t \in T \label{IP:c8}\\
    & \sum_{f = s}^{s+\lambda}\sum_{j \in T \setminus \{i\}}x_{jif} \leq \lambda & \forall i \in T, \ s \in R_{r-\lambda} \label{IP:c9} 
\end{align}
\begin{align}
    & \sum_{f = s}^{s+\lambda}\sum_{j \in T \setminus \{i\}}x_{jif} \geq 1 & \forall i \in T, \ s \in R_{r-\lambda} \label{IP:c10} \\
    & x_{ijs} \in \{0,1\} & \forall i,j \in T: i \neq j, \ s \in R \label{IP:c12}\\
    & y_{tij} \in \{0,1\} & \forall i,j,t \in T \label{IP:c13}
\end{align}

The objective function \eqref{IP:obj} minimizes the total travel distance. Constraints \eqref{IP:c1}-\eqref{IP:c3} find a valid iRR timetable, in which teams play against the same opponent at most once \eqref{IP:c1}, exactly one game in every round \eqref{IP:c2}, and exactly $\frac{r}{2}$ home games \eqref{IP:c3}, and consequently also $\frac{r}{2}$ away games. Constraints \eqref{IP:c4} are linking constraints that state that team $t$ travels from $i$ to $j$ if there exists a round $s$ such that $t$ plays away against $i$ in $s$ and away against $j$ in $s+1$, respectively. Constraints \eqref{IP:c5} enforce a team to travel back home if it ends a road trip while Constraints \eqref{IP:c6} enforce a team to travel from its home venue to the venue of its opponent if it starts a road trip. Constraints \eqref{IP:c7}-\eqref{IP:c8} do the same for the first and last rounds. Constraints \eqref{IP:c9}-\eqref{IP:c10} regulate that teams play at most $\lambda$ consecutive home and away games. Finally, the variable domains are given in \eqref{IP:c12}-\eqref{IP:c13}. 

It is known that the LP-relaxation of this formulation, when applied to the classic TTP, yields a very weak lower bound \citep{ribeiro2012sports}. 
We expect a similar performance in the case of iTTP, based on the intuition that the $x_{ijs}$ variables in constraints \eqref{IP:c4}-\eqref{IP:c6} can be set sufficiently small such that they never trigger the travel variables ($y_{tij}$), rendering only Constraints \eqref{IP:c7}-\eqref{IP:c8} binding. We can therefore strengthen this formulation by adding a lower bound on the minimum number of legs:

\begin{equation}\label{IP:c14}
    \sum_{t \in T}\sum_{i \in T}\sum_{j \in T}y_{tij} \geq v_{nr}
\end{equation}


\subsection{A road trip formulation}\label{sec:HAPdriven}

Similar to the model to compute the DLB, the second formulation we propose ($\mathcal{F}_2$) uses road trip variables. The difference is that we now specify the start date of each road trip. Let $z_{tps}$ be a binary variable that is 1 if team $t$ starts road trip $p\in P_t$ in round $s$ and 0 otherwise. If a team $t$ starts a road trip $p$ in round $s$, we say that $p$ is active in rounds $s, s+1,\dots,s+\delta_{tp}-1$ for team $t$. Recall that $\delta_{tp}$ is the length of road trip $p \in P_t$, and that $R_s = \{1,\dots,r\}, s \in R$. Let $\tau_{tp} = \max \{s \in R \mid s + \delta_{tp} \leq r\}$, i.e. $\tau_{tp}$ is the latest round in which team $t$ can start road trip $p$. Finally, let $p_l$ be the $l^{\text{th}}$ visited team in trip $p$. We define the following sets:

\begin{align}
    \nonumber & A_{ts} = \{(p,q) \ | \ p \in P_t, \ q \in R_{\tau_{tp}}: q \leq s < q+\delta_{tp}\} & \forall t \in T, \ s \in R\\
    \nonumber & V_{its} = \{(p,q) \ | \ p \in P_i: t \in p, \ q \in R_s: q \leq s < q+\delta_{ip} \ \land \ p_{s-q+1} = t\} & \forall t,i \in T: t \neq i, \ s \in R
\end{align}

In other words, $A_{ts}$ denotes the set of road trip-start round pairs $(p,q)$ such that trip $p \in P_t$, when started in round $q$, is active for team $t$ in round $s$. Similarly, $V_{its}$ denotes the set of road trip-start round pairs $(p,q)$ such that team $i$ starts road trip $p$ in round $q$, and visits team $t$ during that road trip in round $s$. Observe that for each triple $i,t,s$, there is exactly one round $q$ for each road trip $p \in P_i$, such that $i$ travels to $t$ in round $s$. Using these sets, we can construct the following formulation:

\begin{align}
    \text{min} & \sum_{t \in T}\sum_{p \in P_{t}}\sum_{s \in R_{\tau_{tp}}}c_{tp}z_{tps} \label{IP-Trips:obj} \\
    \text{s.t.} & \sum_{p \in P_t}\sum_{s \in R_{\tau_{tp}}}\delta_{tp}z_{tps} = \frac{r}{2} & \forall t \in T \label{IP-Trips:c1}\\
    & \sum_{s \in R} \bigl( \sum_{(p,q) \in V_{tis}}z_{tpq} + \sum_{(p,q) \in V_{its}}z_{ipq} \bigr) \leq 1 & \forall t,i \in T: t < i \label{IP-Trips:c2}\\
    & \sum_{(p,q) \in A_{t,s-1}}z_{tpq} + \sum_{p \in P_t: \tau_{tp} \leq s}z_{tps} \leq 1 & \forall t \in T, \ s \in R \setminus \{1\} \label{IP-Trips:c3}\\
    & \sum_{(p,q) \in A_{ts}}z_{tpq} + \sum_{i \in T \setminus \{t\}}\sum_{(p,q) \in V_{its}}z_{ipq} = 1 & \forall t \in T, \ s \in R \label{IP-Trips:c5} \\ 
    & \sum_{f=s}^{s+\lambda}\sum_{i \in T \setminus \{t\}}\sum_{(p,q) \in V_{itf}}z_{ipq} \leq \lambda & \forall t \in T, \ s \in R_{r-\lambda} \label{IP-Trips:c6} \\
    & z_{tps} \in \{0,1\} & \forall t \in T, \ p \in P_t, \ s \in R_{\tau_{tp}} \label{IP-Trips:c7}
\end{align}

The objective minimizes the total travel distance, while Constraints~\eqref{IP-Trips:c1} state that each team should play exactly half of its games away. Constraints~\eqref{IP-Trips:c2} state that two teams can face each other at most once. Constraints~\eqref{IP-Trips:c3} ensure that a team can only start a road trip in round $s$ if it has no active road trip in round $s-1$. Next, Constraints~\eqref{IP-Trips:c5} state that a team $t$ is either on exactly one road trip in round $s$, or it visited by exactly one other team $i$ in round $s$. Hence, it also prevents road trips from overlapping with each other.
Constraints~\eqref{IP-Trips:c6} guarantee that in any $\lambda+1$ consecutive rounds, every team has at least one away game. Note that because of Constraints~\eqref{IP-Trips:c3} and since road trips have length at most $\lambda$, a team automatically does not play more than $\lambda$ consecutive away games. Finally, Expressions~\eqref{IP-Trips:c7} give the domains of the $z$-variables. 

Comparing this formulation with  $\mathcal{F}_1$, we observe that it has $\mathcal{O}(n^2)$ constraints,  while $\mathcal{F}_1$ has $\mathcal{O}(n^4)$ constraints. 
While $\mathcal{F}_1$ has a number of variables polynomial in the number of teams and rounds, $\mathcal{F}_2$ uses $n\bigl(\sum_{f=1}^\lambda \binom{n-1}{f}\bigr)r=\mathcal{O}(n^\lambda r)$ variables. In the Appendix, we prove the following result:

\begin{theorem}\label{thm:LP}
    The LP-relaxation of $\mathcal{F}_2$ is strictly stronger than the LP-relaxation of $\mathcal{F}_1$.
\end{theorem}

How much stronger the LP-relaxation of $\mathcal{F}_2$ is in practice compared to $\mathcal{F}_1$ is investigated computationally in \Cref{sec:Results} in \Cref{tab:DiLB}. 


\subsection{The HAP-set constrained travel minimization problem}\label{sec:HAPconstrained}

Here, we assume the HAPs of teams are given. For this purpose, let $m: T \times R \rightarrow \{H,A\}$ be a function such that $m$ denotes the current assignment of HAPs to teams, and define $m(t,s) = \text{H}$ if team $t$ plays home in round $s$ and $m(t,s) = \text{A}$ if team $t$ plays away in round $s$. 
Observe that if the HAPs are fixed, the rounds in which teams start their road trips as well as the length of the road trips are fixed. Hence, define $R_t^{'}$ as the set of rounds in which team $t$ starts a road trip, and $a_{ts}$ as the number of consecutive away games starting from round $s \in R_t^{'}$ for team $t$. Then, we can define $P_{ts}$ as the set of road trips that team $t$ can start in round $s$. Moreover, we define $W_{ts}$ as the sets of triples $(u,p,q)$ such that team $u$ visits team $t$ in round $s$ when starting road trip $p$ in round $q$. These sets are formally defined as follows:

\begin{align}
    \nonumber & R_t^{'} = \{s \in R \ | \ (s = 1 \ \lor \ m(t,s-1) = \text{H}) \ \land \ m(t,s) = \text{A})\} & \forall t \in T \\
    \nonumber & P_{ts} = \{p \in P_t \ | \ s \leq \tau_{tp}, \ s \in R_t^{'}, \ \delta_{tp} = a_{ts}\} & \forall t \in T, \ s \in R \\
    \nonumber & W_{ts} = \{(u,p,q) \ | \ u \in T: m(u,s) = \text{A}, (p,q) \in V_{uts}: q \in R_u^{'}\} & \forall t \in T, \ s \in R: m(t,s) = \text{H}
\end{align}

Note that $P_{ts} = \emptyset$ in case team $t$ does not start a road trip in round $s$. Using these sets, we construct the following formulation with the road trip variables:

\begin{align}
    \text{min} & \sum_{t \in T}\sum_{s \in R_t^{'}}\sum_{p \in P_{ts}}c_{tp}z_{tps} \label{IP-HAP:obj} \\
    \text{s.t.} & \sum_{p\in P_{ts}} z_{tps} = 1 & \forall t\in T,\ s\in R_t^{'} \label{IP-HAP:c1}\\
    & \sum_{(u,p,q) \in W_{ts}}z_{upq} = 1 & \forall t \in T,\ s\in R: m(t,s) = \text{H}\label{IP-HAP:c2}\\
    & \sum_{s\in R_t^{'}} \sum_{p\in P_{ts}: i\in p} z_{tps} + \sum_{s\in R_i^{'}} 
    \sum_{p\in P_{is}: t\in p} z_{ips} \leq 1 & \forall t,i\in T: t < i\label{IP-HAP:c3}\\
    & z_{tps} \in \{0,1\} & \forall t \in T, \ s\in R_t^{'}, \ p \in P_{ts} \label{IP-HAP:c6}
\end{align}

The objective again minimizes the total travel distance, while Constraints~\eqref{IP-HAP:c1} require that each team is assigned a road trip during all time slots during which it HAP implies the start of a road trip. Constraints~\eqref{IP-HAP:c2} state that each team playing home is visited by one of the teams playing away. Finally, Constraints~\eqref{IP-HAP:c3} state that each pair of teams meets each other at most once and Constraints~\eqref{IP-HAP:c6} are the binary constraints.
We coin the problem of finding an iTTP solution when the HAPs of teams are given as the HAP-set constrained travel minimization problem and refer to the formulation given by Constraints~\eqref{IP-HAP:c1}-\eqref{IP-HAP:c6} as $\mathcal{F}_2$-HAP.
Note that since now we only have to define the $z$-variables for the rounds in which teams start a road trip and the length of this road trip is fixed, this model uses significantly fewer variables than the model given in \Cref{sec:HAPdriven}.

\subsection{Greedy matching algorithm}\label{sec:UpperBounds_greedy_matching}

In \citet{li2025beyond}, a heuristic is proposed to quickly find high quality timetables for iRR tournaments. Although it was developed in the context of youth sports scheduling where teams travel back home after each game, we show that this algorithm can be modified to accommodate our problem. First, we generate a set of feasible HAPs, which we denote as $\mathcal{H}$. In the case of iTTP, a HAP is feasible if it contains at most $\lambda$ consecutive home and away games (satisfying \textbf{C1}), and contains an equal number of home and away games (satisfying \textbf{C3}).
Then, an initial solution is found by assigning to each team a HAP from $\mathcal{H}$ such that in every round, the number of teams that play home is equal to the number of teams that play away. \citet{li2025beyond} propose to generate the complete set of feasible HAPs, and to assign HAPs to teams via an integer program.

Given a home-away assignment, each round can be seen as a minimum weight bipartite matching problem, where each home team should be matched to an away team at minimum cost. \citet{li2025beyond} propose that the rounds are solved sequentially, such that round $s$ contains a bipartite graph $G(T_{s}^h,T_{s}^a,E_s)$ with $T_{s}^h$ the set of home teams in round $s$, $T_{s}^a$ the set of away teams in round $s$, and $E_s$ the set of edges between $T_s^h$ and $T_s^a$ such that there is no edge between two teams that were already matched in an earlier round. 

Let $c_{ijs}$ be the cost of scheduling match $(i,j), i \in T_{s}^h, j \in T_{s}^a$ in round $s$. In the problem considered in \citet{li2025beyond}, teams travel back home after every away game. Therefore, $c_{ijs} = d_{ij}$ and is independent of the matches scheduled in other rounds. In contrast, in iTTP the cost of a match in round $s > 1$ depends on the matching in round $s-1$. Recall that with $m$ we denote the current assignment of HAPs to teams, such that $m(i,s) = \text{H}$ if team $i$ plays home in round $s$ and $m(i,s) = \text{A}$ if team $i$ plays away in round $s$, and let $o(i,s)$ be the opponent of $i$ in $s$. Then, the cost of match $(i,j)$ in round $s$ is as follows:

\begin{align}
c_{ijs} &= d_{ji} & \text{if } (s > 1 \ \land \ m(j,s-1) = \text{H}) \ \lor \ s = 1 \label{c_ijs1}\\
        &+ d_{ij} & \text{if } (s \leq r-1 \ \land \ m(j,s+1) = \text{H}) \ \lor \ s = r \label{c_ijs2}\\
        &+ d_{o(j,s-1),i} & \text{if } s \geq 2 \ \land \ m(j,s-1) = \text{A} \label{c_ijs3}
\end{align}

In other words, the cost $c_{ijs}$ takes into account the cost for $j$ to travel to $i$ in case $j$ has to start from its home venue \eqref{c_ijs1}, together with the cost for $j$ to travel back home from $i$ in case $s$ is the last round or $j$ plays home in the next round \eqref{c_ijs2}, as well as the cost for $j$ to travel from the venue of the previous opponent to $i$ in case $j$ also played away in the previous round \eqref{c_ijs3}.

We now distinguish two approaches. First, with \emph{constructive greedy matching} (GM-c), we denote the algorithm that assigns HAPs to teams only once, and for this assignment, finds a feasible timetable by solving a sequence of minimum weight bipartite matchings. Second, in \citet{li2025beyond}, it is proposed to iteratively modify the HAPs, each time solving a new series of minimum weight perfect matchings, until a convergence criterion is met. We refer to this algorithm as \emph{iterative greedy matching} (GM-it).

\subsubsection{Constructive greedy matching}\label{sec:constructive}

Greedily taking perfect matchings round by round may lead to scenarios where the graph consisting only of unselected matches does not contain a perfect matching, and hence where a timetable of $r$ rounds cannot be constructed \citep{schmand2022greedy}. In this case, the set of selected matches is also known as a premature set of one-factors (see e.g.\ \citet{rosa1982premature}). However, from \citet[Theorem 5.]{li2025beyond} we know that, if all opponents can play each other and $r \leq \frac{n}{4}$, the greedy matching algorithm always finds a feasible timetable, no matter the assignment of HAPs to teams. Moreover, if $|\mathcal{H}| = 2$, the following stronger result holds:

\begin{theorem}\label{thm:Hall}
    Let $r \leq \frac{n}{2}$ and $|\mathcal{H}| = \{h_1,h_2\}$, such that $h_1$ and $h_2$ are complementary HAPs, and all teams are eligible to play against each other. Then, if we assign $h_1$ to one half of the teams and $h_2$ to the other half, the greedy matching algorithm always finds a feasible timetable in which teams with HAP $h_1$ play against teams with HAP $h_2$.
\end{theorem}

\begin{proof}
    Let $T_{h_1}$ be the teams that are assigned HAP $h_1$ and let $T_{h_2}$ be the teams that are assigned HAP $h_2$. Since $h_1$ and $h_2$ are complementary, matches are only possible between teams of $T_{h_1}$ and $T_{h_2}$. Since $|T_{h_1}| = |T_{h_2}| = \frac{n}{2}$ and $r \leq \frac{n}{2}$, the set of all possible matches with these two HAPs can be represented by a bipartite graph $G$ with degree $r$. By Hall's marriage theorem, any regular bipartite graph contains a perfect matching. Since $G$ remains a regular bipartite graph after deleting any sequence of perfect matchings until all $r$ rounds are scheduled, the greedy matching algorithm in this case always finds a feasible timetable.
\end{proof}

Therefore, in case that $r \leq \frac{n}{2}$, it is sufficient to take an arbitrary pair of complementary HAPs. A natural question is whether the constructive greedy matching allows a bounded approximation ratio. Here, we will show with an example that this is not the case. 

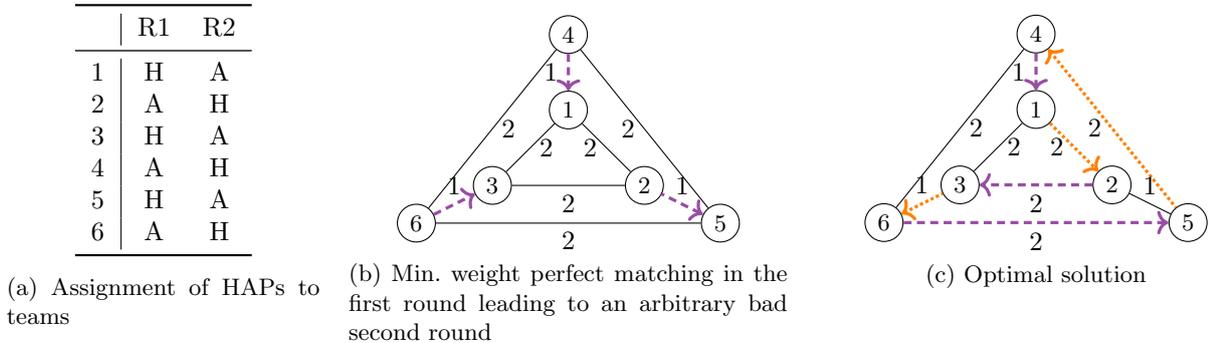
\begin{figure}[ht!]
\caption{\Cref{example1}, which shows that constructive greedy matching does not allow a bounded approximation ratio. Only the matches with a travel distance of 1 or 2 are depicted.}
  \label{fig:example_case1}
\begin{subfigure}[t]{0.25\textwidth}
    \centering
    \begin{minipage}[t][3.55cm][t]{\linewidth} 
        \centering
        \begin{threeparttable}
        \begin{tabular}{c|cc}
        \toprule
             & R1 & R2 \\
             \midrule
            1 & H & A \\
            2 & A & H \\
            3 & H & A \\
            4 & A & H \\
            5 & H & A \\
            6 & A & H \\
            \bottomrule
        \end{tabular}
        \end{threeparttable}
    \end{minipage}
        \caption{Assignment of HAPs to teams}
    \label{fig:HAPS1}
\end{subfigure}
\hspace{4pt}
  \begin{subfigure}[t]{0.35\textwidth}
    \centering
    \begin{tikzpicture}[baseline=(current bounding box.center)]
        \node[default] (w1) at (0,1.5) {$1$};
        \node[default] (w2) at (1,0.5) {$2$};
        \node[default] (w3) at (-1,0.5) {$3$};

        \node[default] (w4) at (0,2.5) {$4$};
        \node[default] (w5) at (2,0) {$5$};
        \node[default] (w6) at (-2,0) {$6$};

        \draw (w1) -- (w2) node[midway, left] {2};
        \draw (w2) -- (w3) node[midway, below] {2};
        \draw (w3) -- (w1) node[midway, right] {2};

        \draw (w4) -- (w5) node[midway, left] {2};
        \draw (w5) -- (w6) node[midway, below] {2};
        \draw (w6) -- (w4) node[midway, right] {2};

        \draw[->, edge4, very thick, densely dashed] (w4) -- (w1) node[midway, left, text=black] {1};
        \draw[<-, edge4, very thick, densely dashed] (w5) -- (w2) node[midway, above, text=black] {1};
        \draw[->, edge4, very thick, densely dashed] (w6) -- (w3) node[midway, above, text=black] {1};
    \end{tikzpicture}
    \caption{Min.\ weight perfect matching in the first round leading to an arbitrary bad second round}
    \label{fig:premature_pm_k_odd}
  \end{subfigure}
  \hspace{4pt}
  \begin{subfigure}[t]{0.35\textwidth}
    \centering
    \begin{tikzpicture}[baseline=(current bounding box.center)]

        \node[default] (w1) at (0,1.5) {$1$};
        \node[default] (w2) at (1,0.5) {$2$};
        \node[default] (w3) at (-1,0.5) {$3$};

        \node[default] (w4) at (0,2.5) {$4$};
        \node[default] (w5) at (2,0) {$5$};
        \node[default] (w6) at (-2,0) {$6$};

        \draw[->, edge5, very thick, densely dotted] (w1) -- (w2) node[midway, left, text=black] {2};
        \draw[->, edge4, very thick, densely dashed] (w2) -- (w3) node[midway, below, text=black] {2};
        \draw (w3) -- (w1) node[midway, right, text=black] {2};

        \draw[<-, edge5, very thick, densely dotted] (w4) -- (w5) node[midway, left, text=black] {2};
        \draw[<-, edge4, very thick, densely dashed] (w5) -- (w6) node[midway, below, text=black] {2};
        \draw (w6) -- (w4) node[midway, right, text=black] {2};

        \draw[->, edge4, very thick, densely dashed]  (w4) -- (w1) node[midway, left, text=black] {1};
        \draw (w5) -- (w2) node[midway, above, text=black] {1};
        \draw[<-, edge5, very thick, densely dotted] (w6) -- (w3) node[midway, above, text=black] {1};
    \end{tikzpicture}
    \caption{Optimal solution}
    \label{fig:optimal_2_factor_k_odd}
  \end{subfigure}
\end{figure}

\begin{example}\label{example1} 
Consider an iTTP instance with $n=6$ and $r=2$. Since the timetable needs to be balanced, each team visits only one other team, after which it immediately travels back home. The travel distance between every pair of teams in the set $\{\{1,4\},\{2,5\},\{3,6\}\}$ is equal to 1, while the travel distance between every pair of teams in the set $\{\{1,2\},\{2,3\},\{3,1\},\{4,5\},\{5,6\},\{6,4\}\}$ is equal to 2. The travel distance between all other pairs of teams is equal to $\alpha$, where $\alpha$ is assumed to be an arbitrary large number. Teams are assigned HAPs according to \Cref{fig:HAPS1}. In this case, the minimum weight perfect matching respecting the HAPs in the first round consists of all matches between teams with a travel distance of 1, as is depicted by the purple dashed arcs in \Cref{fig:premature_pm_k_odd}. However, deleting these matches results in the matches with a travel distance of 2 forming two odd cycles. Hence, in the second round at least one match with a travel distance of $\alpha$ needs to be selected. However, as evidenced by \Cref{fig:optimal_2_factor_k_odd}, there exists a timetable that respects the HAPs and avoids matches with a travel distance of $\alpha$. As $\alpha$ is an arbitrary large number, we have shown that the approximation ratio of constructive greedy matching can be arbitrarily bad.
\end{example}

\subsubsection{Iterative greedy matching}\label{sec:GM_it}

Instead of terminating after a single assignment of HAPs to teams, \citet{li2025beyond} propose to iteratively modify the HAPs with moves in four different neighborhood structures. \Cref{{example1}} suggests that this is indeed recommended if we want to obtain high quality solutions. Two of the four neighborhoods proposed in \citet{li2025beyond} could be used for iTTP: a move in the neighborhood \emph{RandomSwap} swaps the HAPs of two random teams and a move in the neighborhood \emph{ComplementInsertion} replaces the patterns of two teams with complementary HAPs with two other complementary HAPs. However, these neighborhood structures assume that the complete set of HAPs is given. If there are few restrictions on the HAPs, as is the case in iTTP, the number of possible HAPs quickly explodes as $r$ and $\lambda$ increase. For example, for $\lambda=3$ and $r=30$, there are at least 208,606,320 distinct patterns. Even more problematic, is the fact that by only using RandomSwap and ComplementInsertion, it can be impossible to reach certain HAPs. For example, consider the assignment of HAPs to teams as shown in \Cref{tab:example2_HAA}.  It can be verified that \Cref{tab:example2_HAA} does not contain a pair of complementary HAPs and that \Cref{tab:example2_opp_schedule} is a feasible timetable such that each team plays home and away according to \Cref{tab:example2_HAA}. In this case, only the move RandomSwap can be used, making it impossible to reach certain HAPs.

\begin{table}[h!]
    \small
    \setlength{\tabcolsep}{5pt}
\caption{Timetable with no pair of teams with complementary HAPs}
\label{tab:example2}
\begin{subtable}[t]{0.45\textwidth}
    \centering
\begin{threeparttable}
\begin{tabular}{c|ccc ccc}
\toprule
& R1 & R2 & R3 & R4 & R5 & R6 \\
\midrule
1 & H & H & A & A & H & A \\
2 & A & H & H & A & A & H \\
3 & A & A & A & H & H & H \\
4 & H & A & H & H & A & A \\
5 & H & H & A & A & H & A \\
6 & A & H & H & A & A & H \\
7 & A & A & A & H & H & H \\
8 & H & A & H & H & A & A \\
\bottomrule
\end{tabular}
\end{threeparttable}
\caption{Home-away assignment}
\label{tab:example2_HAA}
\end{subtable}
\hspace{0.45cm}
\begin{subtable}[t]{0.45\textwidth}
\centering
\begin{threeparttable}
\begin{tabular}{cccccc}
\toprule
R1 & R2 & R3 & R4 & R5 & R6 \\
\midrule
1 vs 2 & 1 vs 3 & 4 vs 1 & 7 vs 1 & 1 vs 8 & 6 vs 1 \\
4 vs 3 & 2 vs 7 & 2 vs 3 & 8 vs 2 & 7 vs 4 & 2 vs 4 \\
5 vs 6 & 5 vs 4 & 8 vs 5 & 3 vs 5 & 3 vs 6 & 3 vs 8 \\
8 vs 7 & 6 vs 8 & 6 vs 7 & 4 vs 6 & 5 vs 2 & 7 vs 5 \\
\bottomrule
\end{tabular}
\end{threeparttable}
\caption{Timetable}
\label{tab:example2_opp_schedule}
\end{subtable}
\end{table}

Therefore, we use a new neighborhood structure to modify the HAPs, which we call \emph{2-Team HomeAwaySwap} (2THAS). This neighborhood is connected (see further) and does not require an explicit enumeration of HAPs. First, we define an assignment of HAPs to teams $m$ to be \emph{proper} if and only if, for every round $s \in R$, it holds that $|\{i \in T: m(i,s) = \text{H}\}| = |\{i \in T: m(i,s) = \text{A}\}| = \frac{n}{2}$. Furthermore, we call $m$ balanced if and only if for every team $i \in T$, it holds that $|\{s \in R: m(i,s) = \text{H}\}| = |\{s \in R: m(i,s) = \text{A}\}| = \frac{r}{2}$. Then, 2THAS takes as input 2 distinct teams $i,j \in T$ and 2 distinct rounds $q,s \in R$ such that $m(i,s) \neq m(j,s)$, $m(i,s) \neq m(i,q)$ and $m(i,q) \neq m(j,q)$. A 2THAS move involves swapping $m(i,s)$ for $m(j,s)$ and $m(i,q)$ for $m(j,q)$. An example of a sequence of moves in 2THAS is shown in \Cref{tab:2THAS}, where home-away swaps belonging to the same move are denoted by the same color.

\begin{table}[h!]
    \small
    \setlength{\tabcolsep}{5pt}
\caption{Transformation of $m$ to $m^{\star}$ and $m'$ to $m^{\star}$ with 2THAS moves}
\label{tab:2THAS}
\begin{subtable}[t]{0.32\textwidth}
    \centering
\begin{tabular}{c|cccc}
\toprule
& R1 & R2 & R3 & R4 \\
\midrule
1 & \typeD{H} & \typeA{A} & \typeA{H} & \typeD{A} \\
2 & \typeB{A} & \typeD{H} & \typeB{H} & \typeD{A} \\
3 & \typeD{H} & \typeC{A} & \typeD{A} & \typeC{H} \\
4 & \typeB{H} & \typeD{A} & \typeB{A} & \typeD{H} \\
5 & \typeD{A} & \typeC{H} & \typeD{H} & \typeC{A} \\
6 & \typeD{A} & \typeA{H} & \typeA{A} & \typeD{H} \\
\bottomrule
\end{tabular}
\caption{HAP assignment $m$}
\label{tab:2THAS1}
\end{subtable}
\begin{subtable}[t]{0.32\textwidth}
\centering
\begin{threeparttable}
\begin{tabular}{c|cccc}
\toprule
& R1 & R2 & R3 & R4 \\
\midrule
1 & \typeD{H} & \typeD{H} & \typeD{A} & \typeD{A} \\
2 & \typeD{H} & \typeD{H} & \typeD{A} & \typeD{A} \\
3 & \typeD{H} & \typeD{H} & \typeD{A} & \typeD{A} \\
4 & \typeD{A} & \typeD{A} & \typeD{H} & \typeD{H} \\
5 & \typeD{A} & \typeD{A} & \typeD{H} & \typeD{H} \\
6 & \typeD{A} & \typeD{A} & \typeD{H} & \typeD{H} \\
\bottomrule
\end{tabular}
\end{threeparttable}
\caption{HAP assignment $m^{\star}$}
\label{tab:2THAS2}
\end{subtable}
\begin{subtable}[t]{0.32\textwidth}
\centering
\begin{threeparttable}
\begin{tabular}{c|cccc}
\toprule
& R1 & R2 & R3 & R4 \\
\midrule
1 & \typeD{H} & \typeD{H} & \typeD{A} & \typeD{A} \\
2 & \typeD{H} & \typeD{H} & \typeD{A} & \typeD{A} \\
3 & \typeA{A} & \typeC{A} & \typeA{H} & \typeC{H} \\
4 & \typeA{H} & \typeD{A} & \typeA{A} & \typeD{H} \\
5 & \typeD{A} & \typeD{A} & \typeD{H} & \typeD{H} \\
6 & \typeD{A} & \typeC{H} & \typeD{H} & \typeC{A} \\
\bottomrule
\end{tabular}
\end{threeparttable}
\caption{HAP assignment $m'$}
\label{tab:2THAS3}
\end{subtable}
\end{table}

We first argue that if $m$ is proper and balanced, a move in this neighborhood always exists. Indeed, since in any round half the teams play home and half the teams play away, there are exactly $\frac{n^2}{8}$ pairs of teams $\{i,j\}$ in any round $q \in R$ such that $m(i,q) \neq m(j,q)$. Moreover, since each HAP is balanced, it holds that for any such pair there exists a round $s \in R: s \neq q$ such that $m(i,s) \neq m(i,q)  \text{and} \ m(i,s) \neq m(j,s)$. Next, we show that 2THAS is able to transform any proper balanced assignment of HAPs to teams $m$ to any other proper balanced assignment of HAPs to teams $m'$ (i.e., 2THAS is `connected').

\begin{theorem}
    Any proper balanced assignment of HAPs to teams $m$ can be transformed into any other proper balanced assignment of HAPS to teams $m'$ by a finite sequence of moves in the neighborhood structure 2THAS.
\end{theorem}

\begin{proof}
    For ease of notation, we partition the set of teams $T$ into $T_1 = \{1,\dots,\frac{n}{2}\}, \ T_2 = \{\frac{n}{2}+1,\dots,n\}$ and the set of rounds $R$ into $R_1 = \{1,\dots,\frac{r}{2}\}, \ R_2 = \{\frac{r}{2}+1,\dots,r\}$. Without loss of generality, we can assume that $m(i,1) = \text{H}$ for all $i \in T_1$ and $m(j,1) = \text{A}$ for all $j \in T_2$. In order to prove the theorem, we show that starting from any proper balanced assignment of HAPs to teams $m$, we can transform $m$ into $m^{\star}$, which is defined as follows (see also \cref{tab:2THAS2}):
    
    \begin{align*}
    m^{\star}(i,s) =&
    \begin{cases}
        \text{H} & \text{if} \ i \in T_1 \ \text{and} \ s \in R_1 \\
        \text{A} & \text{if} \ i \in T_1 \ \text{and} \ s \in R_2 \\
        \text{A} & \text{if} \ i \in T_2 \ \text{and} \ s \in R_1 \\
        \text{H} & \text{if} \ i \in T_2 \ \text{and} \ s \in R_2 \\
    \end{cases}
    & \forall i \in T, \ \forall s \in R
    \end{align*}  

    Consider the smallest value for $k < \frac{r}{2}$, such that $m(i,k) = \text{H}$ and $m(j,k) = \text{A}$ for all $i \in T_1, \ j \in T_2$, and that there exists a team $u \in T_1$ such that $m(u,k+1) = \text{A}$. Because $m$ is proper, there must exist a team $v \in T_2$ such that $m(v,k+1) = \text{H}$. Since $m$ is balanced, there necessarily must exist two rounds $q,s > k+1$ such that $m(u,q) = \text{H}$ and $m(v,s) = \text{A}$. In case that $s=q$, we can perform a 2THAS move that results in $m(u,k+1) = m(v,q) = \text{H}$ and $m(v,k+1) = m(u,q) = \text{A}$. 
    Now suppose that $s \neq q$. In case that $m(v,q) = \text{A}$, we can do a 2THAS move that results in $m(u,k+1) = m(v,q) = \text{H}$ and $m(v,k+1) = m(u,q) = \text{A}$. 
    Similarly, if $m(u,s) = \text{H}$, we can do a 2THAS move that results in $m(u,k+1) = m(v,s) = \text{H}$ and $m(v,k+1) = m(u,s) = \text{A}$. This leaves the case where $m(v,q) = \text{H}$ and $m(u,s) = \text{A}$. Since $m$ is proper, exactly $\frac{n}{2}$ teams from the set $T \setminus \{u,v\}$ must play away in round $q$, and since $u$ and $v$ play away in round $s$, there are exactly two teams in the set $T \setminus \{u,v\}$ that play away in round $q$ and home in round $s$. Arbitrarily label one of these teams $w$. Then, we first do the 2THAS move including teams $v$ and $w$ and rounds $s$ and $q$, which results in $m(v,q) = \text{A}$, so we can now do a 2THAS move involving teams $u$ and $v$ and rounds $k+1$ and $q$. We can iteratively do this until $m$ and $m^{\star}$ are identical for round $k+1$, after which we proceed with the next round, until $k=\frac{r}{2}$. An example is shown in \Cref{tab:2THAS1}, where the HAP assignment in \Cref{tab:2THAS1} is transformed to $m^{\star}$ by the blue, green an purple 2THAS moves. 

    If for each round $k \leq \frac{r}{2}$, it holds that $m(i,k) = \text{H}$ and $m(j,k) = \text{A}$ for all $i \in T_1, \ j \in T_2$, it follows from the balancedness of the HAPs that $m(i,s) = \text{A}$ and $m(j,s) = \text{H}$ for all $i \in T_1, \ j \in T_2$ and $s > \frac{r}{2}$. Consequently, $m=m^{\star}$. 
    
    Since any two proper balanced HAP assignments $m$ and $m'$ can be transformed into $m^{\star}$, it follows that we can first transform $m$ into $m^{\star}$, and then transform $m^{\star}$ into $m'$, and visa versa. For example, in \Cref{tab:2THAS} $m$ can be transformed into $m^{'}$ by first applying the sequence of 2THAS moves shown in \Cref{tab:2THAS1}, and consequently into $m'$ by applying the moves needed to transform $m'$ into $m^{\star}$ (shown in \Cref{tab:2THAS3}), in opposite order.
\end{proof}

Our version of the iterative greedy matching algorithm works as follows. First, we construct a set of HAPs that serves as a starting solution. In case that $r \leq \frac{n}{2}$, we randomly generate pairs of complementary HAPs with at most $\lambda$ consecutive home and away games, and randomly assign these HAPs to the teams. At first sight, \Cref{thm:travel_distance_CON} suggests that taking complementary HAPs with a maximum number of breaks might be beneficial. However, we obtained better results in our experiments by sampling arbitrary pairs of complementary HAPs. From \Cref{thm:Hall}, we know that greedily taking perfect matchings that respect the HAPs always result in a feasible timetable. In case that $r > \frac{n}{2}$, we follow the approach outlined in \Cref{sec:iTTP} to obtain an intial set of HAPs. While greedily taking perfect matchings is not guaranteed to result in a timetable, at least we know that this HAP set allows one. 

With this HAP, we start the iterative greedy matching algorithm. In each iteration of greedy matching, we modify the HAPs with 2THAS, and redo the greedy matching. If 2THAS results in a HAP with more than $\lambda$ consecutive home or away games, we reject the move and sample new moves until a feasible move has been found. We use hill climbing to navigate through the search space. Hence, a timetable is accepted only if it yields at least the same objective value as the best found solution so far. If the move is not accepted, we undo the HAP modification, and sample a new move. The algorithm continues until we reach 10,000 consecutive iterations without finding a solution that is at least as good as the current best solution. Experiments revealed that allowing more consecutive iterations without improvement did not change the results significantly.

\section{Computational experiments}\label{sec:Results}

We now perform some computational experiments, where we compute lower and upper bounds for iTTP instances. We first discuss how we generate instances, and after briefly discussing the implementation details, we present the results.

\subsection{Instances}

Observe that a TTP instance can readily be transformed to an iTTP instance by specifying an additional parameter $r$, which denotes the number of rounds. Therefore, we take a set of TTP instances available at \url{https://robinxval.ugent.be/robinx/travelRepo.php}. It is known that they are computationally challenging to solve, as evidenced by the fact that none of the instances with more than 10 teams have been solved except for CON, for which an optimal solution with 16 teams is known. As is common in the TTP literature, we specify $\lambda=3$.
The instances and their main characteristics are given in \Cref{tab:instancesTTP}. We refer to an iTTP instance by the instance class it belongs to (its competition name or topological structure) and the number of teams and rounds.
In particular, the NL (\cite{easton2001traveling}), NFL (\cite{uthus2009dfs}), and BRA (\cite{ribeiro2007heuristics}) instances are based on the National League in Major League Baseball and the National Football League and Brazilian football competition, respectively.
In GAL (\cite{uthus2012solving}), instances place artificial teams in a 3D-coordinate space and the distance between them is based on the number of light years between stars in the universe.
Further, in CIRC teams are placed on a circle, while in LINE and INCR \citep{hoshino2012generating} instances teams are located on a straight line. Finally, in CON instances \citep{urrutia2006maximizing} the travel distance between every pair of teams is one (i.e., their objective is minimizing the number of legs). For an instance with $n$ teams, we generate three instances with $r \in \{\frac{n}{4}, \frac{n}{2}, \frac{3n}{4}\}$. Since $n \in \{16,24,32,40\}$, this always results in an even number of rounds.

\begin{table}[ht!]
  \centering
  \small
    \setlength{\tabcolsep}{5pt}
  \caption{Overview of the iTTP instances and their characteristics}
  \label{tab:instancesTTP}
  \begin{threeparttable}
  \begin{tabular}{lcr p{9cm}}
    \toprule
    Setting & No. teams & No.rounds & Characteristics \\
    \midrule
    NL & 16 & 4,8,12 & National League of Major League Baseball \\
    BRA & 24 & 6,12,18 & Brazilian soccer championship \\
    NFL & 32 & 8,16,24 & National Football League \\
    GAL & 40 & 10,20,30 & Teams are located in a three-dimensional `galaxy' \\
    CIRC & 40 & 10,20,30 & Teams are located on a circle \\
    CON & 40 & 10,20,30 & All distances are 1\\
    LINE & 40 & 10,20,30 & Teams are located on a line\\
    INCR & 40 & 10,20,30 & Teams are located on a line with increasing distances\\
    \bottomrule
  \end{tabular}
  \end{threeparttable}
\end{table}

\subsection{Implementation details}

All experiments were run on a GNU/Linux based system with an AMD EPYC 7532 32-Core Processor running at 3.3GHz, provided with 8 threads and 64GB of RAM. All integer programming formulations and their LP-relaxations are solved with Gurobi 12.0.

Lower bounds for CON instances where derived by the formula of \Cref{thm:travel_distance_CON} in case $r \leq \frac{n}{2}$ and by Benders decomposition otherwise. 
To compute the ILB, we used formulation \eqref{lb2:obj}-\eqref{lb2:c6} except that we formulated it only for a single team, enforcing that the total length of all the road trips is equal to $\frac{r}{2}$ and that each opponent is present in at most one selected trip (rather than solving the PCVRP, see \Cref{subsec:ilb}). 
Solutions for the ILB can be found almost instantly. Finding solutions for the DLB is more challenging, especially for larger instances. The computation time of the DLB, DLB-1F and DLB-MinLeg is limited to 48 hours. For the upper bounds, the computation time of solving formulations $\mathcal{F}_1$ and $\mathcal{F}_2$ is also limited to 48 hours, while the computation time of the iterative greedy matching algorithm and $\mathcal{F}_2$-HAP is limited to 12 hours. 

The greedy matching heuristic was implemented in C++ and minimum weight bipartite perfect matchings were found with the Boost Maximum Weighted Matching library (\url{https://www.boost.org/doc/libs/latest/libs/graph/doc/maximum_weighted_matching.html}), which runs in $\mathcal{O}(n^3)$ time. Every time we solve an instance with the greedy matching algorithm, we do so with 100 different seeds. In case that $r > \frac{n}{2}$, we find an initial timetable with Vizing's edge coloring algorithm as explained in \Cref{sec:iTTP}. For this, we use the Boost Edge Coloring library (\url{https://www.boost.org/doc/libs/latest/libs/graph/doc/edge_coloring.html}). Since this algorithm is deterministic, in order to obtain different timetables we randomly shuffle the teams based on the seed.

\subsection{Results}

In this section, we present lower and upper bounds for the instances shown in \Cref{tab:instancesTTP}. \Cref{tab:DiLB} shows the lower bounds. For each instance, we provide the LP-relaxation of formulations $\mathcal{F}_1$ (LP$_{\mathcal{F}_1}$, \Cref{sec:UpperBounds_IP}) and $\mathcal{F}_2$ (LP$_{\mathcal{F}_2}$, \Cref{sec:HAPdriven}), the Independent Lower Bound (ILB, \Cref{subsec:ilb}), the Dependent Lower Bound (DLB, \Cref{subsec:dlb}), the DLB with 1-factorable constraints (DLB-1F, \Cref{sec:coloring_constraints}), and the DLB with the MinLeg constraint (DLB-MinLeg, \Cref{subsec:mnt}). For each instance, the strongest lower bound is marked in bold, while the weakest lower bound is italicized. 

\begin{table}[t]
    \centering
    \small
    \setlength{\tabcolsep}{5pt}
    \caption{Results lower bounds}
    \label{tab:DiLB}
    \begin{threeparttable}
    \begin{tabular}{l rrrrrr}
        \toprule
        Instance & LP$_{\mathcal{F}_1}$ & LP$_{\mathcal{F}_2}$ & ILB & DLB & DLB-MinLeg & DLB-1F \\
        \midrule
        NL16-4 & \emph{9002} & 22696 & 15273 & 23625 & - & \textbf{24464} \\
        NL16-8 & \emph{13598} & 52905 & 36130 & \textbf{57263} & - & \textbf{57263}\\
        NL16-12 & \emph{20760} & 91346 & 55568 & \textbf{92580} & \textbf{92580} & \textbf{92580}\\
        BRA24-6 & \emph{6205.5} & 34982 & 21418 & 38745 & - & \textbf{39127}\\
        BRA24-12 & \emph{6205.5} & 96031 & 51633 & \textbf{98815*} & - & \textbf{98815*}\\
        BRA24-18 & \emph{6205.5} & 165432 & 87123 & \textbf{167657*} & \textbf{167657*} & \textbf{167657*}\\
        NFL32-8 & \emph{10275} & 64912 & 47506 & 70127 & - & \textbf{70158} \\
        NFL32-16 & \emph{19857} & 169263 & 113319 & \textbf{175051*} & - & \textbf{175051*}\\
        NFL32-24 & 36216 & -1 & 191083 & 297178* & \textbf{297305*} &  \textbf{297305*}\\
        CIRC40-10 & \emph{280} & \textbf{560} & 400 & \textbf{560} & - & \textbf{560} \\
        CIRC40-20 & 560 & -1 & 1120 & \textbf{1760} & \textbf{1760} & 1760\\
        CIRC40-30 & 810 & -1 & 2160 & 3600 & \textbf{3620*} & \textbf{3620*}\\
        GAL40-10 & 5261 & -1 & 15546 & \textbf{23794} & - & \textbf{23794} \\
        GAL40-20 & 9915 & -1 & 33055 & \textbf{51992*} & - & \textbf{51992*}\\
        GAL40-30 & 15378.5 & -1 & 52236 & \textbf{81752*} & \textbf{81752*} & \textbf{81752*}\\
        INCR40-10 & 1342.5 & -1 & 8066 & \textbf{13492} & - & \textbf{13492}\\
        INCR40-20 & 2509.5 & -1 & 22892 & \textbf{47222*} & - & \textbf{47222*}\\
        INCR40-30 & 4031.5 & -1 & 45484 & \textbf{102336*} & \textbf{102336*} & \textbf{102336*}\\ 
        LINE40-10 & \emph{283} & 641 & 412 & \textbf{678} & - & \textbf{678}\\
        LINE40-20 & 563 & -1 & 1188 & \textbf{2364} & - & \textbf{2364}\\
        LINE40-30 & 813 & -1 & 2376 & \textbf{5120*} & \textbf{5120*} & \textbf{5120*}\\
        \bottomrule
    \end{tabular}
    \begin{tablenotes}
\item Asterisks (*) denote bounds solved within a 0.1\% optimality gap; all other bounds were solved to proven optimality. For each instance, the best found upper bound is marked in bold, while the worst found upper bound is italicized. A -1 indicates that the model could not be fit into 64GB of memory, or the solver ran out of memory before finding a feasible solution. In - in the column DLB-MinLeg denotes that this bound was not computed.
\end{tablenotes}
\end{threeparttable}
\end{table}

First, we observe that the DLB results in a significant improvement of the bounds compared to the ILB. The DLB-MinLeg improved the DLB only in the case of NFL32-24 and CIRC40-30, so the effect of adding the MinLeg Constraint \eqref{lb2:c8} seems limited. Note that we only compute the DLB-MinLeg for cases where $r > \frac{n}{2}$, since for $r \leq \frac{n}{2}$ this bound cannot improve the DLB (see \Cref{subsec:mnt}). We also observe that adding the 1-factorable Constraints~\eqref{lb2:c9}-\eqref{lb2:c10} improved the lower bound only in the case of NL16-4, BRA24-6 and NFL32-8. In line with the one-factorization conjecture (see \Cref{sec:coloring_constraints}), this suggests that the DLB-1F is only effective in case the number of rounds is relatively small. Next, we turn to the LP-relaxations of $\mathcal{F}_1$ and $\mathcal{F}_2$. As expected, solving the LP-relaxation of $\mathcal{F}_1$ results in a very weak lower bound compared to the other methods, even when Constraint~\eqref{lb2:c8} is added. However, the LP relaxation of $\mathcal{F}_2$ turns out to be relatively strong. Notably, the DLB and its variants still provide tighter bounds than the LP-relaxation of $\mathcal{F}_2$. Moreover, for instances with $n=40$ and $r \in \{20,30\}$, $\mathcal{F}_2$ could either not be fit into 64GB of memory, or the solver ran out of memory before finding a feasible solution. In these cases, column generation is advised for solving the LP-relaxation.

\begin{table}[t]
    \centering
    \small
    \setlength{\tabcolsep}{5pt}
    \caption{Computation times for solution strategies}
    \label{tab:Time}
    \begin{threeparttable}
    \begin{tabular}{l rrrrr}
        \toprule
& $\mathcal{F}_1$ & $\mathcal{F}_2$ & GM-c & GM-it & $\mathcal{F}_2$-HAP \\ 
\midrule
Min & 48.0h & 12.0s & $< 1.0$s & $< 1.0$s & $< 1.0$s \\
Max & 48.0h & 48.0h & $< 1.0$s & 52.5m &  12.0h \\
Average & 48.0h & 36.0h & $< 1.0$s & 9.22m & 6.5h \\
\bottomrule
    \end{tabular}
    \begin{tablenotes}
\item s: seconds, m: minutes, h: hours
\end{tablenotes}
\end{threeparttable}
\end{table}

We now turn to the upper bounds. \Cref{tab:Time} shows the minimum, maximum, and average computation time of each of the proposed methods. Solving formulation $\mathcal{F}_1$ always took the maximum allotted time, even for the smallest instances. While an optimal solution for NL16-4 can found almost instantly by solving $\mathcal{F}_2$, for other instances the solver either ran out of memory or hit the time limit of 48 hours. GM-c found a solution always instantly, while GM-it required up to 52.5 minutes to converge. Finally, for larger instances $\mathcal{F}_2$-HAP is not able to find a proven optimal solution (given a fixed assignment of HAPs to teams) within a time limit of 12 hours, while for smaller instances it is often able to find an optimal solution within a few hours. 

\begin{table}[t]
    \centering
    \small
    \setlength{\tabcolsep}{2pt} 
    \caption{Results upper bounds}
    \label{tab:Results1}
    \begin{threeparttable}
    \begin{tabularx}{\textwidth}{l @{\extracolsep{\fill}} r @{\extracolsep{\fill}} rr @{\extracolsep{\fill}} rr @{\extracolsep{\fill}} rr @{\extracolsep{\fill}} rrr @{\extracolsep{\fill}} rr}
        \toprule
        Instance & LB & \multicolumn{2}{c}{$\mathcal{F}_1$} & \multicolumn{2}{c}{$\mathcal{F}_2$} & \multicolumn{2}{c}{GM-c} & \multicolumn{3}{c}{GM-it} & \multicolumn{2}{c}{$\mathcal{F}_2$-HAP} \\ 
         \cmidrule(lr){3-4} \cmidrule(lr){5-6} \cmidrule(lr){7-8} \cmidrule(lr){9-11} \cmidrule(lr){12-13}
        & & UB & (\%) & UB & (\%) & UB & (\%) & UB & (\%) & Avg UB & UB & (\%) \\
        \midrule
        NL16-4 & 24464 & 25200 & 2.92 &  \textbf{\underline{25076}} & 2.44 & 28815 & 15.1 & \emph{26112} & 6.31 & 28517 & 25876 & 5.46 \\ 
NL16-8 & 57263 & \emph{70448} & 18.72 &  \textbf{62125} & 7.83 & 78027 & 26.61 & 66763 & 14.23 & 73147 & 65068 & 12.0 \\ 
NL16-12 & 92580 & \emph{153583} & 39.72 & \textbf{110037} & 15.86 & 133114 & 30.45 & 120310 & 23.05 & 131165 & 115089 & 19.56 \\ 
BRA24-6 & 39127 & \emph{53361} & 26.67 & \textbf{41004} & 4.58 & 46780 & 16.36 & 43491 & 10.03 & 47214 & 42508 & 7.95 \\ 
BRA24-12 & 98815 & \emph{165017} & 40.12 & -1 & -1 & 151619 & 34.83 & 127438 & 22.46 & 138358 & \textbf{120615} & 18.07 \\ 
BRA24-18 & 167657 & \emph{362904} & 53.8 & -1 & -1 & 290842 & 42.35 & 236738 & 29.18 & 253081 & \textbf{212852} & 21.23 \\ 
NFL32-8 & 70158 & \emph{115763} & 39.4 & -1 & -1 & 98153 & 28.52 & 85201 & 17.66 & 91932 & \textbf{83771} & 16.25 \\ 
NFL32-16 & 175051 & \emph{362382} & 51.69 & -1 & -1 & 274387 & 36.2 & 235752 & 25.75 & 255240 & \textbf{215826} & 18.89 \\ 
NFL32-24 & 297305 & \emph{651580} & 54.37 & -1 & -1 & 542675 & 45.21 & 443518 & 32.97 & 467803 & \textbf{396264} & 24.97 \\ 
CIRC40-10 & 560 & \emph{1812} & 69.09 & -1 & -1 & 1064 & 47.37 & 852 & 34.27 & 929 & \textbf{814} & 31.2 \\ 
CIRC40-20 & 1760 & \emph{640000} & 99.72 & -1 & -1 & 4306 & 59.13 & 2940 & 40.14 & 3278 & \textbf{2770} & 36.46 \\ 
CIRC40-30 & 3620 & \emph{10054} & 63.99 & -1 & -1 & 7472 & 51.55 & 5974 & 39.4 & 6330 & \textbf{5640} & 35.82 \\ 
CON40-10 & 280 & \emph{329} & 14.89 & -1 & -1 & 286 & 2.1 & \textbf{280} & 0.0 & 280 & \textbf{280} & 0.0 \\ 
CON40-20 & 560 & \emph{659} & 15.02 & -1 & -1 & 620 & 9.68 & \textbf{560} & 0.0 & 608 & \textbf{560} & 0.0 \\ 
CON40-30 & 810 & \emph{974} & 16.84 & -1 & -1 & 993 & 18.43 & \textbf{812} & 0.25 & 816 & \textbf{812} & 0.25 \\ 
GAL40-10 & 23794 & \emph{49632} & 52.06 & -1 & -1 & 33353 & 28.66 & 28872 & 17.59 & 30270 & \textbf{28201} & 15.63 \\ 
GAL40-20 & 51992 & \emph{102362} & 49.21 & -1 & -1 & 87529 & 40.6 & 66713 & 22.07 & 72432 & \textbf{63135} & 17.65 \\ 
GAL40-30 & 81752 & \emph{154593} & 47.12 & -1 & -1 & 140553 & 41.84 & 120554 & 32.19 & 124393 & \textbf{112270} & 27.18 \\ 
INCR40-10 & 13492 & \emph{41892} & 67.79 & -1 & -1 & 23526 & 42.65 & 18890 & 28.58 & 20564 & \textbf{17796} & 24.19 \\ 
INCR40-20 & 47222 & \emph{173689} & 72.81 & -1 & -1 & 165864 & 71.53 & 77374 & 38.97 & 88317 & \textbf{67078} & 29.6 \\ 
INCR40-30 & 102336 & \emph{256898} & 60.16 & -1 & -1 & 212798 & 51.91 & 166360 & 38.49 & 175668 & \textbf{142404} & 28.14 \\ 
LINE40-10 & 678 & \emph{3578} & 81.05 & -1 & -1 & 1192 & 43.12 & 936 & 27.56 & 1036 & \textbf{884} & 23.3 \\ 
LINE40-20 & 2364 & \emph{8334} & 71.63 & -1 & -1 & 6458 & 63.39 & 3908 & 39.51 & 4369 & \textbf{3370} & 29.85 \\ 
LINE40-30 & 5120 & \emph{13432} & 61.88 & -1 & -1 & 10904 & 53.04 & 8244 & 37.89 & 8796 & \textbf{6990} & 26.75 \\ 
\midrule
Average & & & 48.78 & & 7.68 & & 37.53 & & 24.11 & & & 19.60 \\
\bottomrule
    \end{tabularx}
    \begin{tablenotes}
For each instance, the best found upper bound is marked in bold, while the worst found upper bound is italicized. The value for NL16-4 found by $\mathcal{F}_2$ is optimal. 
\end{tablenotes}
    \end{threeparttable}
\end{table}

Finally, \Cref{tab:Results1} shows the results. The first column repeats the best found lower bound (LB) from \Cref{tab:DiLB}. Next, it shows the best found upper bound, together with the gap, found by solving $\mathcal{F}_1, \mathcal{F}_2$, one iteration of constructive greedy matching (GM-c), iterative greedy matching (GM-it) and $\mathcal{F}_2$-HAP, respectively, for each instance. The average gap achieved by each method over all the instances (for which it could find a solution) is also reported. For GM-c and GM-it, the best upper bound over the 100 different seeds is reported. Moreover, for GM-it, we also show the average upper bound. From \Cref{tab:Results1}, we can see that the solutions found by $\mathcal{F}_1$ are relatively poor, even for small instances like NL16-8 and BRA24-6. The solutions found for instances of settings CIRC, INCR and LINE are observed to be particularly poor, which is most likely due to the high amount of symmetry in these instances. For most instances, formulation $\mathcal{F}_2$, which contains more than 68 million variables, could not be fit into memory. However, for NL16-4 it is able to obtain an optimal solution in just a few seconds, and for instances NL16-8, NL16-12 and BRA24-6, it finds the best solution (within 48 hours) among all methods. Notably, our strongest lower bound for these three instances is stronger than the lower bound that GUROBI found after solving $\mathcal{F}_2$ for 48 hours.

Concerning the greedy matching heuristic, modifying the HAPs seem to be the preferred strategy, as the best found solution of GM-it is often around 10\% better than the best found solution by GM-c. Moreover, $\mathcal{F}_2$-HAP is able to further improve upon the best found solution by GM-it, achieving gaps between 0\%-36\%. Nevertheless, the gaps reported in \Cref{tab:Results1} are still relatively high, even for instances with a small number of teams and/or a limited number of rounds. This suggests that both the construction of better lower and upper bounds, respectively, are promising directions for future research. 

\section{Conclusion}\label{sec:Conclusion}

In this paper, we introduce the incomplete Traveling Tournament Problem (iTTP), which extends the well-known Traveling Tournament Problem to the context of incomplete round-robin tournaments.
We gave a formal problem description, proposed problem instances, and provided several methods to obtain lower bounds and solutions inspired by existing formulations and algorithms. Although the problem appears similar to the traditional TTP, the additional requirement of determining which teams will play against each other introduces an additional challenge. Our works suggests that approaches based solely on either TTP or iRR research struggle to efficiently solve iTTP problems even small-sized problem instances to optimality. However, by combining ideas from both worlds, we managed to obtain some promising first results. We hope that our contribution will trigger further research from the academic community to devise stronger bounds and more efficient solution strategies.

\bibliographystyle{apalike}
\bibliography{bibliography}

\vspace*{-10pt}

\section*{Appendix A}\label{App:B}

Here, we prove \Cref{thm:LP}, which states that the LP-relaxation of formulation $\mathcal{F}_{2}$ is strictly stronger than that formulation $\mathcal{F}_{1}$.

\begin{proof}

Let LP$_{\mathcal{F}}(I)$ be the optimal objective value found by solving the LP-relaxation of a formulation $\mathcal{F}$ for instance $I$. For a minimization problem, the LP-relaxation of formulation $\mathcal{F}_{2}$ is said to be strictly stronger than that of another formulation $\mathcal{F}_{1}$ if and only if the following conditions hold:

\begin{enumerate}
    \item for each instance $I$, LP$_{\mathcal{F}_2}(I)$ $\geq$ LP$_{\mathcal{F}_1}(I)$,
    \item there exists an instance $I$ such that LP$_{\mathcal{F}_2}(I)$ $>$ LP$_{\mathcal{F}_1}(I)$.
\end{enumerate}

From \Cref{tab:DiLB}, we can see that the second condition is satisfied for all instances considered in this paper. In order to prove the first condition, we need to show that any solution of the LP-relaxation of $\mathcal{F}_2$ corresponds to a feasible solution of the LP-relaxation of $\mathcal{F}_1$ with the same objective function value. Given a solution for solution of the LP-relaxation of $\mathcal{F}_2$, we set the variables in the LP-relaxation of $\mathcal{F}_1$ as follows: 

\begin{align}\label{eq:setx}
& x_{its} = \sum_{(p,q) \in V_{tis}}z_{tpq} & \forall (i,t) \in M, s \in R
\end{align}

In order to set the $y$-variables, we let $L_{tp}$ be the set of legs that constitute road trip $p$ for team $t$, i.e.\ $L_{tp} = \{(t,p_1), (p_{\delta_{tp}},t)\} \cup \{(p_i, p_{i+1}): i \in \{1,\dots,\delta_{pt}-1\}\}$. We set 

\begin{align}\label{eq:sety}
& y_{tij} = \sum_{p \in P_t: (i,j) \in L_{tp}}\sum_{s \in R}z_{tps} & \forall i,j,t \in T
\end{align}

We now show that this solution satisfies all the constraints from $\mathcal{F}_1$. First, observe that the following equalities holds:

\begin{align}\label{eq:helper1}
   & \sum_{(p,q) \in A_{ts}}z_{tpq} = \sum_{i \in T \setminus \{t\}}\sum_{(p,q) \in V_{tis}}z_{tpq} & \forall t \in T, s \in R
\end{align}

Equality~\eqref{eq:helper1} follows from the fact that if $t$ has an active trip in round $s$, then by the definition of the sets $V$, there is exactly one team that is visited by team $t$ in round $s$. 

Because of Constraint~\eqref{IP-Trips:c5}, the values of the $x$-variables are bounded by $[0,1]$. Because of Constraints~\eqref{IP-Trips:c2} and \eqref{IP-Trips:c7}, the values of the $y$-variables are bounded by $[0,1]$. This settles the domain constraints (\ref{IP:c12}) and (\ref{IP:c13}).

We proceed with the constraints that contain only $x$-variables. Constraints~\eqref{IP:c1} are satisfied since for any pair $i,j \in T: i < j$:

\begin{align}
    & \sum_{s \in R}(x_{ijs}+x_{jis}) = \sum_{s \in R} \bigl( \sum_{(p,q) \in V_{jis}}z_{jpq} + \sum_{(p,q) \in V_{ijs}}z_{ipq} \bigr) \stackrel{\eqref{IP-Trips:c2}}{\leq} 1 
\end{align}

Constraints~\eqref{IP:c2} are satisfied since for any $i \in T, \ s \in R$:

\begin{align}
    \sum_{j \in T \setminus \{i\}}(x_{ijs}+x_{jis}) =& \sum_{j \in T \setminus \{i\}} \bigl( \sum_{(p,q) \in V_{jis}}z_{jpq} + \sum_{(p,q) \in V_{ijs}}z_{ipq} \bigr) \\
    \stackrel{\eqref{IP-Trips:c5}}{=}& 1 - \sum_{(p,q) \in A_{is}}z_{ipq} + \sum_{j \in T \setminus \{i\}}\sum_{(p,q) \in V_{ijs}}z_{ipq} \\
    \stackrel{\eqref{eq:helper1}}{=}& 1 - \sum_{j \in T \setminus \{i\}}\sum_{(p,q) \in V_{ijs}}z_{ipq} + \sum_{j \in T \setminus \{i\}}\sum_{(p,q) \in V_{ijs}}z_{ipq} = 1
\end{align}

Next, Constraints~\eqref{IP:c3} are also satisfied since for any $i \in T$:

\begin{align}
    \sum_{j \in T \setminus \{i\}}\sum_{s \in R}x_{ijs} = \sum_{j \in T \setminus \{i\}}\sum_{s \in R}\sum_{(p,q) \in V_{jis}}z_{jpq} = \sum_{p \in P_t}\sum_{s \in R_{\tau_{tp}}}\delta_{tp}z_{tps} \stackrel{\eqref{IP-Trips:c1}}{=} \frac{r}{2}
\end{align}

This follows from the fact that any value for $z_{tps}$ is present $\delta_{tp}$ times in the sum $\sum_{j \in T \setminus \{i\}}\sum_{s \in R}\sum_{(p,q) \in V_{jis}}z_{jpq}$ (once for each team $j$ visited by team $i$ in its road trip $p$).
\vspace{2pt}
Constraints~\eqref{IP:c9} are satisfied as well for any $i \in T, s \in R_{r-\lambda}$:

\begin{align}
    \sum_{f = s}^{s+\lambda}\sum_{j \in T \setminus \{i\}}x_{jif} = \sum_{f = s}^{s+\lambda}\sum_{j \in T \setminus \{i\}}\sum_{(p,q) \in V_{ijf}}z_{ipq} 
    \stackrel{\eqref{IP-Trips:c6}}{\leq} \lambda
\end{align}

Constraints~\eqref{IP:c10} also hold, since for any $i \in T, s \in R_{r-\lambda}$:

\begin{align}
    \sum_{f = s}^{s+\lambda}\sum_{j \in T}x_{jif} & = \sum_{f = s}^{s+\lambda} \bigl( \sum_{j \in T \setminus \{i\}}\sum_{(p,q) \in V_{ijf}}z_{ipq} \bigr) \\
    & \stackrel{\eqref{eq:helper1}}{=} \sum_{f =s}^{s+\lambda} \sum_{(p,q) \in A_{if}}z_{ipq} \stackrel{\eqref{IP-Trips:c5}}{=} \sum_{f =s}^{s+\lambda}\bigl(1 - \sum_{j \in T \setminus \{i\}}\sum_{(p,q) \in V_{jif}}z_{jpq} \bigr) \\
    & \stackrel{\eqref{IP-Trips:c6}}{\geq} \lambda+1-\lambda = 1
\end{align}

We continue by showing that the constructed $x$ and $y$ variables satisfy Constraints~\eqref{IP:c4}-\eqref{IP:c8}. Constraints~\eqref{IP:c4} are satisfied since for all $i,j,t \in T$ and any $s \in R \setminus \{1\}$ it holds that:

\begin{align}
    x_{its-1} + x_{jts}-1 = \sum_{(p,q) \in V_{tis-1}}z_{tpq} + \sum_{(p,q) \in V_{tjs}}z_{tpq}-1 \leq \max \{0, \sum_{p \in P_t: (i,j) \in L_{tp}}\sum_{s \in R}z_{tps}\} \leq y_{tij}
\end{align}

Next, the constructed solution satisfies Constraints~\eqref{IP:c5} since:

\begin{align}
    x_{its-1} + \sum_{j \in T \setminus \{i\}}x_{tjs}-1 & = \sum_{(p,q) \in V_{tis-1}}z_{tpq} + \sum_{j \in T \setminus \{i\}}\sum_{(p,q) \in V_{jts}}z_{jpq}-1 \\
    & \leq \max \{0, \sum_{p \in P_t: (i,t) \in L_{tp}}\sum_{s \in R}z_{tps}\} \leq y_{tit}
\end{align}

By a similar argument, Constraints~\eqref{IP:c6} are also satisfied. Constraints~\eqref{IP:c7} are satisfied for all $t,i \in T$ since:

\begin{align}
    x_{it1} = \sum_{(p,q) \in V_{ti1}}z_{tpq} \leq \sum_{s \in R}\sum_{p \in P_t: (t,i) \in L_{pt}}z_{tps} = y_{tit}
\end{align}

Similarly, Constraints~\eqref{IP:c8} are also seen to be satisfied for all $t,i \in T$:

\begin{align}
    x_{itr} = \sum_{(p,q) \in V_{tir}}z_{tpq} \leq \sum_{s \in R}\sum_{p \in P_t: (i,t) \in L_{pt}}z_{tps} = y_{tti}
\end{align}


Finally, the objective function value for the $x$ and $y$ variables constructed via \Cref{eq:setx} and \Cref{eq:sety} is equal to that of the original solution of the LP-relaxation of formulation $\mathcal{F}_{2}$:

\begin{align}
    \sum_{t \in T}\sum_{i \in T}\sum_{j \in T}d_{ij}y_{tij} &= \sum_{t \in T}\sum_{i \in T}\sum_{j \in T}d_{ij}\bigl(\sum_{p \in P_t: (i,j) \in L_{tp}}\sum_{s \in R}z_{tps}\bigr) \\
    & = \sum_{t \in T}\sum_{p \in P_t}\sum_{s \in R}\bigl( \sum_{(i,j) \in L_{tp}}d_{ij} \bigr)z_{tps} = \sum_{t \in T}\sum_{p \in P_t}\sum_{s \in R}c_{tp}z_{tps}
\end{align}


\end{proof}

\end{document}